\documentclass[11pt]{amsart}
\usepackage{amscd,amssymb}
\usepackage{amsthm,amsmath}

\usepackage{float}
\usepackage{pgf,tikz}
\usepackage{mathrsfs}
\usetikzlibrary{arrows}

\usepackage{fancyhdr}
\usepackage{supertabular}
\usepackage{setspace}
\usepackage{color}
\usepackage{longtable}

\usepackage[matrix,arrow,curve]{xy}
\usepackage{graphicx}
\usepackage{enumerate}
\usepackage{amsfonts}
\usepackage{cite}
\usepackage{pifont}
\usepackage{multirow}
\usepackage{hhline}

\newtheorem{theorem}[equation]{Theorem}
\newtheorem*{maintheorem*}{Main Theorem}
\newtheorem{lemma}[equation]{Lemma}
\newtheorem{proposition}[equation]{Proposition}
\newtheorem{corollary}[equation]{Corollary}
\newtheorem{conjecture}[equation]{Conjecture}

\newtheorem*{question*}{Question}

\newtheorem*{problem*}{Problem}

\theoremstyle{definition}

\newtheorem{definition}[equation]{Definition}
\newtheorem{remark}[equation]{Remark}

\theoremstyle{remark}

\makeatletter\@addtoreset{equation}{section} \makeatother

\definecolor{sqsqsq}{rgb}{0.12549019607843137,0.12549019607843137,0.12549019607843137}
\definecolor{cqcqcq}{rgb}{0.7529411764705882,0.7529411764705882,0.7529411764705882}

\newcommand{\mult}{\operatorname{mult}}

\title{K-stability of smooth del Pezzo surfaces}

\author{Jihun Park and Joonyeong Won}

\address{ \emph{Jihun Park}\newline \textnormal{Center for Geometry and Physics, Institute for Basic Science (IBS)
\newline \medskip 77 Cheongam-ro, Nam-gu, Pohang, Gyeongbuk, 37673, Korea. \newline Department of Mathematics, POSTECH \newline 77 Cheongam-ro, Nam-gu, Pohang, Gyeongbuk, 37673, Korea. \newline \texttt{wlog@postech.ac.kr}}}%

\address{\emph{Joonyeong
Won}\newline \textnormal{Center for Geometry and Physics, Institute for Basic Science (IBS)
\newline \medskip 77 Cheongam-ro, Nam-gu, Pohang, Gyeongbuk, 37673, Korea. \newline \texttt{leonwon@kias.re.kr}}}%

\thanks{This work has been supported by IBS-R003-D1, Institute for Basic Science in Korea. 
The authors are grateful to Kento Fujita and Yuji Odaka  who brought the article \cite{FO16} to their attention.  
They also thank Giulio Codogni who informed them of the article \cite{K78}. }

\begin{document}

\begin{abstract}
In a new algebro-geometric way we completely determine whether  smooth del Pezzo surfaces are K-(semi)stable or not.
\end{abstract}

\maketitle

In the present article, all varieties are defined over an algebraically closed field $k$  of characteristic $0$.
\section{Introduction}
\label{sec:intro}
Since entering the 21st century we have   witnessed dramatic developments in the study of the Yau-Tian-Donaldson conjecture concerning the existence of K\"ahler-Einstein metrics on Fano manifolds and stability. The challenge to the  conjecture has been highlighted  by Chen, Donaldson, Sun and Tian who have completed the proof of the following celebrated statement (\cite{CDS15a}, \cite{CDS15b}, \cite{CDS15c}, \cite{T15}).

\begin{theorem}\label{theorem:CDST}
Let $X$ be a smooth Fano variety defined over $\mathbb{C}$. 
It admits  a K\"ahler-Einstein metric if and only if the pair  $(X, -K_X)$ is K-polystable.
\end{theorem}

The motivation of this article cannot  be expressed  in a better way than quoting the following phrase in one of the three articles by   Chen, Donaldson and Sun (\cite{CDS15c}):
\begin{quote}
``On the other hand, we should point out that as things stand at present the result is of very limited use in concrete cases, so that there is no manifold $X$ known to us, not covered by other existence results and where we can deduce that $X$ has a K\"ahler-Einstein metric. This is because it seems a very difficult matter to test K-stability by a direct study of all possible degenerations. However, we are optimistic that this situation will change in the future, with a deeper analysis of the stability condition.''
\end{quote}

There are not so many results concerning K-stability of specified smooth Fano varieties, not deduced by K\"ahler-Einstein metrics. It seems almost infeasible to consider all possible degenerations of a given Fano manifold. Even for del Pezzo surfaces, we do not have complete description of their degenerations. There are a classification only for the projective plane and some partial classifications for the others (\cite{Ma91}, \cite{Ma93}, \cite{HPr10}, \cite{Pr15}). We also have a  classification of del Pezzo surfaces with 
quotient singularities and Picard
rank one that admit $\mathbb{Q}$-Gorenstein smoothings (\cite{HPr10}). These are however not enough to directly test K-stability of del Pezzo surfaces. 

Meanwhile, since asymptotic Chow-stability implies K-semistability (for instance, see \cite{RTh07}), an algebro-geometric proof for the K-semistability of the projective spaces  can be yielded by the celebrated result of Kempf in \cite{K78} that a homogeneous rational variety embedded with a complete linear system is Chow-stable.

There are a few of algebro-geometric  methods known to us that can be utilized to prove K-stability in concrete cases. One of the ways  is based on the $\alpha$-invariant originally introduced by Tian (\cite{T87}). 
The original definition of the $\alpha$-invariant was given  in an analytic way.  However, there is an algebro-geometric   way to define the $\alpha$-invariant over an arbitrary field of characteristic zero.
 \begin{definition}
Let $X$ be a Fano orbifold defined over $k$. The global log canonical threshold of $X$
is defined by the number
\[\alpha(X)=\sup\left\{\lambda\in\mathbb{Q}\ \left|%
\aligned
&\text{the  pair}\ \left(X, \lambda D\right)\ \text{is log canonical for every effective} \\
&\text{ $\mathbb{Q}$-divisor $D$ numerically equivalent to $-K_X$.}\\
\endaligned\right.\right\}.%
\]
\end{definition}
It  has been verified that for a Fano orbifold  defined over $\mathbb{C}$ its $\alpha$-invariant coincides with the global log canonical threshold (\cite{ChS08}).
For this reason, the same name $\alpha$-invariant and  the same notation $\alpha(X)$  will be used for the global log canonical threshold of the Fano orbifold defined over $k$ in the present article.

The original purpose of the $\alpha$-invariant is to show the existence of K\"ahler-Einstein metrics on given Fano manifolds (\cite{T87}).
Fujita, Odaka and Sano however reinterpret the $\alpha$-invariant as a sufficient condition for a $\mathbb{Q}$-Fano variety to be K-stable. 

 \begin{theorem}[\cite{OS12}, \cite{F16}]\label{theorem:FOS}
Let $X$ be a $\mathbb{Q}$-Fano variety. Either if 
\begin{equation}\label{inequality}
 \alpha(X) >\frac{\dim (X)}{\dim(X)+1} 
\end{equation}
or if  $\alpha(X) =\frac{\dim (X)}{\dim(X)+1}$ and $X$ is smooth,
then the pair $(X,-K_X)$ is K-stable.
\end{theorem}

This enables us to have a  detour studying   K-stability for some specific $\mathbb{Q}$-Fano varieties.
Even though the method is only a one-side implication and the $\alpha$-invariants are not easy to compute at all, 
this is the only pragmatic way, known to us so far,  to verify K-stability in concrete cases.  It is an indisputable expectation that computing the $\alpha$-invariant of a given Fano variety should be much more doable than directly investigating its degenerations.

For instance, Cheltsov has computed the exact values of the $\alpha$-invariants of all the smooth del Pezzo surfaces. For the purpose of K-stability, we summarize  his computation  in the following way.

\begin{theorem}[\cite{Ch08}] Let $S$ be a smooth del Pezzo surface of degree $d$.
\begin{itemize}
\item If $d\geq 5$, then $\alpha(S)<\frac{2}{3}$.
\item If $d\leq 4$, then $\alpha(S)\geq\frac{2}{3}$.
\end{itemize}
\end{theorem}
With this estimation, the method of Fujita-Odaka-Sano immediately implies that a smooth del Pezzo surface of degree at most $4$
 is K-stable with respect to its anticanonical polarisation.
 
Even though we are not able to  completely determine the K-stability of all the smooth del Pezzo surfaces with their $\alpha$-invariants, this result  demonstrates that the $\alpha$-invariant is a very practical tool to test K-stability for given Fano varieties.  The following higher dimensional smooth Fano varieties are also instructive examples to which we can apply the method of Fujita-Odaka-Sano: 
\begin{itemize}
\item a  double cover of $\mathbb{P}^n$ ramified along a smooth hypersurface of degree $2n$, $n\geq 3$;
\item a  smooth hypersurface of degree $n$ in $\mathbb{P}^n$, $n\geq 4$.
\end{itemize}
All the smooth Fano varieties  in these families satisfy the condition of Theorem~\ref{theorem:FOS} (\cite{CP10}, \cite{CP02}). Therefore, they are K-stable, and hence the Fano manifolds defined over $\mathbb{C}$ in these families admit  K\"ahler-Einstein metrics. In particular, a smooth hypersurface of degree $n$ in $\mathbb{P}^n$ with generalized Eckardt points (\cite{CP02})  is an example of a K\"ahler-Einstein Fano manifold  whose K\"ahler-Einstein metric is verified to exist only by proving its K-stability (\cite{CP02}, \cite{CPW14}, \cite{F16}).

Recently Fujita and Odaka provided a new algebro-geometric  way to test K-(semi)stability of Fano varieties. To introduce their method, let $X$ be a $\mathbb{Q}$-factorial variety with at worst log canonical singularities,
$Z\subset X$ a closed subvariety and $D$ an effective
 $\mathbb{Q}$-divisor on $X$. The log canonical threshold
of $D$ along $Z$ is the number
$$c_Z(X,D)=\mathrm{sup}\left\{c\ \Big|\ \mbox{the  pair } (X, cD)
 ~\text{is log canonical  along}~Z.\right\}.$$
Because log canonicity is a local property, we see that
$$c_Z(X,D)=\inf_{p\in Z} \left\{c_p(X,D)\right\}.$$
If $X=\mathbb{A}^n$ and $D=(f=0)$, where $f$ is a polynomial defined over $\mathbb{A}^n$, then we also
 use the notation $c_0(f)$ for the log canonical threshold of
 $D$ at the origin.

\begin{definition}[\cite{FO16}]
Let $X$ be a $\mathbb{Q}$-Fano variety and let $m$ be a positive integer such that the plurianticanonical linear system $|-mK_X|$ is non-empty. Set $\ell_m=h^0(X,\mathcal{O}_X(-mK_X))$.  For a section $s$ in $\mathrm{H}^0(X,\mathcal{O}_X(-mK_X))$, we denote the effective divisor of the section $s$ by $D(s)$.
If $\ell_m$ sections $s_1,\ldots, s_{\ell_m}$  form  a basis of  the space $\mathrm{H}^0(X,\mathcal{O}_X(-mK_X))$, then the  anticanonical 
$\mathbb{Q}$-divisor 
\[D:=\frac{1}{\ell_m}\sum_{i=1}^{\ell_m}\frac{1}{m}D (s_i)\]
is said to be of $m$-basis type.
For a positive integer $m$, we set 
\[\delta_m(X)=\inf_{\underset{\mbox{$m$-basis type}}{D:}}c_X(X,D).\]
We set $\delta_m(X)=0$ if $|-mK_X|$ is empty.
The $\delta$-invariant of $X$ is defined by the number
\[\delta(X)=\limsup_m \delta_m(X).\]
\end{definition}

Using the $\delta$-invariant, Fujita and Odaka  set up a conjectural criterion for K-(semi)stability.
They then proved that the $\delta$-invariant gives a sufficient condition for K-(semi)stability.

\begin{conjecture}[\cite{FO16}]\label{conjecture}
A $\mathbb{Q}$-Fano variety $X$ is K-stable (resp. K-semistable) with respect to $-K_X$ if and only if $\delta(X)>1$ (resp. $\geq 1$).
\end{conjecture}
Note that Conjecture~\ref{conjecture} is true if Berman-Gibbs stability (\cite{F16a})  is equivalent to K-stability.

\begin{theorem}[\cite{FO16}]\label{theorem:sufficient}
Let $X$ be a $\mathbb{Q}$-Fano variety. If $\delta(X)>1$ (resp. $\geq 1$), then $(X, -K_X)$ is K-stable (resp. K-semistable).
\end{theorem}
\begin{proof}
See \cite{FO16}.
\end{proof}

In the present article we utilize this theorem to verify the K-(semi)stability of smooth del Pezzo surfaces, i.e., we prove the following.
\begin{maintheorem*}\label{theorem:main}
Let $S$ be a smooth del Pezzo surface of degree $d$. 
\begin{itemize}
\item If $d\leq 5$, then  $\delta(S)\geq\frac{15}{14}>1$.
\item If $S\cong \mathbb{P}^2$, $\mathbb{P}^1\times\mathbb{P}^1$ or the del Pezzo surface of degree $6$, then $\delta(S)=1$.
\item If $S\cong\mathbb{F}_1$ or the del Pezzo surface of degree $7$, then $\delta(S)<1$.
\end{itemize}
\end{maintheorem*}

Through Theorem~\ref{theorem:sufficient}, these estimations of the $\delta$-invariants immediately  yield the following.

\begin{corollary}\label{corollary:stable}
Let $S$ be a smooth del Pezzo surface of degree $d$. Then the pair 
\[(S, -K_{S}) \mbox{ is } \left\{\begin{array}{ll}
\mbox{K-semistable  if $S\cong \mathbb{P}^2$, $\mathbb{P}^1\times\mathbb{P}^1$ or the del Pezzo surface of degree $6$.}\\
\mbox{K-stable    if $d\leq 5$.}\\
\end{array}\right.\]
\end{corollary}

Since the sufficient condition for K-stability given by the $\alpha$-invariant  in \eqref{inequality} involves the dimensions of Fano varieties,
the following theorem shows that the $\alpha$-invariant  itself is paralyzed in testing K-stability of products of Fano varieties.

\begin{theorem}
Let $X$  and $Y$ be smooth Fano varieties.  Then 
$$\alpha (X\times Y)=\min\left\{ \alpha (X), \alpha(Y)\right\}.$$
\end{theorem}
\begin{proof}
See \cite[Lemma~2.29]{ChS08}. 
\end{proof}
We strongly believe that the same formula holds for the $\delta$-invariant.
\begin{conjecture}\label{conjecture:product}
Let $X$  and $Y$ be smooth Fano varieties.  Then 
$$\delta (X\times Y)=\min\left\{ \delta (X), \delta(Y)\right\}.$$
\end{conjecture}
\begin{remark}
One can easily verify the conjecture in case when either $X$ or $Y$ is $1$-dimensional, i.e., $\mathbb{P}^1$.
\end{remark}
From this formula we see that 
the $\delta$-invariant, unlike the $\alpha$-invariant,  keeps its ability to test K-stability  in products of Fano varieties. For instance, Main Theorem implies the following.

\begin{corollary}
Assume that Conjecture~\ref{conjecture:product} holds. Then a product of smooth del Pezzo surfaces of degrees $\leq 5$ is K-stable. A product of K-semistable   smooth del Pezzo surfaces  is K-semistable.
\end{corollary}

It is certain that the $\delta$-invariant produces a new method to verify existence of K\"ahler-Einstein metrics on Fano varieties through Theorem~\ref{theorem:CDST}. It is no exaggeration to say that this new method is simpler and  more coherent than the method by the $\alpha$-invariant. 

\begin{corollary}
A smooth del Pezzo surface of degree $\leq 5$ defined over $\mathbb{C}$ admits a K\"ahler-Einstein metric.
\end{corollary}

Since Conjecture~\ref{conjecture} has not been verified completely, at this moment we cannot say that the last statement of Main Theorem implies that the Hirzebruch surface $\mathbb{F}_1$
and the del Pezzo surface of degree $7$ are not K-semistable. 
Yet it should be remarked here that  algebro-geometric ways to prove their non-K-semistability are already known.
\begin{proposition}\label{proposition:non-stable}
The del Pezzo surfaces of degree $7$ and the Hirzebruch surface $\mathbb{F}_1$ are not K-semistable.
\end{proposition}
\begin{proof}
For instance, see \cite[Examples~6.5 and 6.5]{F15}, \cite[Proposition~3.1]{HKLP11}, \cite{PR09}
\end{proof}

Since the del Pezzo surfaces in the second statement of Main Theorem are all toric varieties, they cannot be K-stable. Therefore, Corollary~\ref{corollary:stable} and Proposition~\ref{proposition:non-stable} verify that Conjecture~\ref{conjecture} holds good for 2-dimensional Fano manifolds.

\section{Preliminaries}
Let $f$ be a polynomial over the field  $k$ in variables $z_1,\ldots, z_n$. Assign positive integral weights $w(z_i)$ to the variables $z_i$. Let $w(f)$ be the weighted multiplicity of $f$ at the origin ($=$ the lowest weight of the monomials occurring in $f$) and let $f_w$ denote the weighted homogeneous leading term of $f$ ($=$ the term of the monomials  in $f$ with the  weighted multiplicity of $f$).

 Let $g$ be a polynomial over the field $k$ in  $z_2,\ldots, z_n$ and set
\[h(z_1,\ldots,z_n)=f(z_1+g(z_2,\ldots, z_n), z_2,\ldots, z_n).\]
If $z_1+g(z_2,\ldots, z_n)$ is weighted homogeneous with respect to the given weights $w(z_1),\ldots, w(z_n)$, it is clear that
\[h_w(z_1,\ldots,z_n)=f_w(z_1+g(z_2,\ldots, z_n), z_2,\ldots, z_n).\]
Let $f_1,\ldots, f_\ell$ be polynomials  over the field $k$ in  $z_1,\ldots, z_n$. With respect to the given weights $w(z_1),\ldots, w(z_n)$, we easily see that
\[\left(\prod_{i=1}^{\ell}f_i\right)_w=\prod_{i=1}^{\ell}(f_i)_w, \  \ \  w\left(\prod_{i=1}^{\ell}f_i\right)=\sum_{i=1}^{\ell}w(f_i).\]

Let $\mathfrak{m}\subset k[z_1,\ldots, z_n]$ be the maximal ideal of the origin in $\mathbb{A}^n$. Let $f_1,\ldots, f_\ell$ be polynomials  over the field $k$ in  $z_1,\ldots, z_n$ that induce a basis for the $d$-jet space at the origin, i.e., $k[z_1,\ldots, z_n]/\mathfrak{m}^{d+1}$, where $d$ is a positive integer and $\ell=\dim_k k[z_1,\ldots, z_n]/\mathfrak{m}^{d+1}$. The $k$-linear map 
of $k[z_1,\ldots, z_n]/\mathfrak{m}^{d+1}$ induced by the coordinate change $z_1-g(z_2,\ldots, z_n)\mapsto z_1$ and $z_i\mapsto z_i$ for $i\geq 2$ is an automorphism. Therefore, the new polynomials $f_1(z_1+g(z_2,\ldots, z_n), z_2,\ldots, z_n),\ldots,
f_\ell(z_1+g(z_2,\ldots, z_n), z_2,\ldots, z_n)$ also induce a basis for the $d$-jet space.

\begin{proposition}\label{c_0}
Let $f$ be a polynomial over $\mathbb{A}^n$. Assign integral
weights $w(z_i)$ to the variables and let $w(f)$ be the weighted multiplicity of $f$.  Then
\begin{itemize}
\item $c_0(f_w)\leq c_0(f)\leq \frac{\sum w(z_i)}{w(f)}$.
\item  If  $$\left(\mathbb{A}^n, \frac{\sum w(z_i)}{w(f)}\cdot \left(f_w=0\right)\right)$$ is log canonical outside the origin, then 
$ c_0(f)=\frac{\sum w(z_i)}{w(f)}$.
\end{itemize}
\end{proposition}
\begin{proof}
See \cite[Propositions~8.13 and 8.14]{K97}.
\end{proof}

\section{K-semistable del Pezzo surfaces I}\label{section:semi-I}

Fix a positive integer $m$ and put $\ell_m =h^0(\mathbb{P}^2, \mathcal{O}_{\mathbb{P}^2}(-mK_{\mathbb{P}^2}))$.  Let $[x:y:z]$ be a homogeneous coordinate for $\mathbb{P}^2$.  

Since the $\ell_m$ monomials of degree $3m$ in $x$, $y$, $z$ form a basis of the space $\mathrm{H}^0(\mathbb{P}^2, \mathcal{O}_{\mathbb{P}^2}(-mK_{\mathbb{P}^2}))$, the effective divisor defined by the equation $xyz=0$ is an anticanonical divisor of $m$-basis type. This shows that
\[\delta_m(\mathbb{P}^2)\leq 1\]
for each $m$. Therefore, $\delta(\mathbb{P}^2)\leq 1$.

\begin{theorem}\label{theorem:d=9}
The $\delta$-invariant of $\mathbb{P}^2$ is $1$.
\end{theorem}

\begin{proof}

Let $\{s_1,\ldots, s_{\ell_m}\}$ be a basis of $\mathrm{H}^0(\mathbb{P}^2, \mathcal{O}_{\mathbb{P}^2}(-mK_{\mathbb{P}^2}))$. 
We denote the effective divisor of the section $s_i$ by $D_i$ and set
\[D:=\sum_{i=1}^{\ell_m}D_i.\]

It is enough to show that for an arbitrary point $p$ in $\mathbb{P}^2$  $$c_p(\mathbb{P}^2, D)\geq \frac{1}{m\ell_m}.$$

By a linear coordinate change, we may assume that $p=[0:0:1]$.

Consider the $\ell_m$ monomials of degree $3m$ in the variables $x, y, z$.  Putting $z=1$, we obtain 
 the $\ell_m$ monomials of degrees at most $3m$ $$1,x,y,x^2,xy,xy^2,\ldots, x^k,x^{k-1}y, \ldots, xy^{k-1},y^k,\ldots,  x^{3m},x^{3m-1}y, \ldots, xy^{3m-1},y^{3m},$$ where they are written in  lexicographic order. 
In this order, denote the $i$-th monomial by $\mathbf{x}_i$ for each $i=1,\ldots, \ell_m$.

 We may consider $x$ and~$y$ as local coordinates around the point $p$.
Then each $D_i$ is defined around the point $p$ by a polynomial $f_i$ of degree at most $3m$ in  the variables $x, y$.  The divisor $D$ is defined by $f:=\prod f_i$ in an affine neighborhood of $p$.

Since the sections $s_1,\ldots, s_{\ell_m}$ form a basis for $\mathrm{H}^0(\mathbb{P}^2, \mathcal{O}_{\mathbb{P}^2}(-mK_{\mathbb{P}^2}))$ and they induce the polynomials $f_1,\ldots, f_{\ell_m}$, respectively, we may assume that $f_i$ contains the monomial $\mathbf{x}_i$ for each $i=1,\ldots, \ell_m$.

Now we consider the Newton polygon of the polynomial $f$ in $\mathbb{R}^2$, where we use coordinate functions $(s,t)$ for $\mathbb{R}^2$.
\bigskip

\textbf{Claim 1.} The Newton polygon of the polynomial $f$ contains the point $(m\ell_m, m\ell_m)$ corresponding to the monomial~$x^{m\ell_m}y^{m\ell_m}$.
\medskip

Since each $f_i$ contains the monomial $\mathbf{x}_i$, we have $w(f_i)\leq w(\mathbf{x}_i)$ with respect to  given weights $w(x), w(y)$, and hence $w(f)\leq w(\prod \mathbf{x}_i)=w(x^{m\ell_m}y^{m\ell_m})$. This proves the claim.
\bigskip

\textbf{Claim 2.} If the line $s=t$ intersects the Newton polygon of $f$ at one of its vertices, then~$c_0(f)\geq \frac{1}{m\ell_m}$. 
\medskip

In such a case, we can take weights $w(x), w(y)$ such that $f_w$ is of the form $\epsilon x^a y^a$, where $\epsilon$ is a non-zero constant. Claim~1 implies that $a\leq m\ell_m$. We then obtain $$c_0(f)\geq c_0(f_w)=\frac{1}{a}\geq \frac{1}{m\ell_m}$$
from Proposition~\ref{c_0}.

\bigskip

We may therefore assume that any vertex of the Newton polygon of $f$ does not lie on the line~$s=t$.

\textbf{Step A.} Let $\Lambda$ be the edge of the Newton polygon of $f$ that intersects the line $s=t$. 
Let~$w(x), w(y)$ be the weights  such that   the monomials of $f_w$ are plotted  on the edge~$\Lambda$.
 
 If $\Lambda$ is either vertical ( $f_w=x^ag(y)$) or horizontal ($f_w=y^ah(x)$), where $g$ and $h$ are polynomials of multiplicity at most $a-1$ at the origin, then Claim~1 implies $a\leq m\ell_m$. It then follows from Proposition~\ref{c_0} that $c_0(f)\geq \frac{1}{m\ell_m}$.

Therefore, we may assume that $\Lambda$ is neither vertical nor horizontal. The slope of the edge  $\Lambda$ is equal to $-\frac{w(x)}{w(y)}$.

\bigskip

\textbf{Step B.}
We write an irreducible decomposition of $f_w$  as $$f_w=\epsilon x^ay^b(x^{\alpha_1}+g_1(x,y))^{c_1} \cdots(x^{\alpha_r}+g_r(x,y))^{c_r},$$ 
where $\epsilon$ is a non-zero constant and $g_i(x,y)$ is an irreducible  weighted homogeneous polynomial of degree $w(x^{\alpha_i})$ such that it does not contain the monomial $x^{\alpha_i}$.
Note that  $a,b\leq m\ell_m$. 

\bigskip

Let $c=\max\{c_i\}.$  We may assume that $c_1=c$.  For the convenience, we set $\alpha=\alpha_1$. Since~$g_1(x,y)$ is irreducible,  it must  possess  the monomial $y^\beta$ for some positive integer $\beta$.

\bigskip

\textbf{Claim 3.}
If $c\leq m\ell_m$, then $c_0(f)\geq \frac{1}{m\ell_m}.$

\bigskip
It immediately follows from Proposition~\ref{c_0} that 
$$c_0(f_w)=\mathrm{min}\left\{\frac{1}{a},\frac{1}{b},\frac{1}{c},\frac{w(x)+w(y)}{w(f_w)}\right\}.$$
By Claim~1,  $\frac{w(x)+w(y)}{w(f_w)}\geq \frac{1}{m\ell}$, so that $c_0(f_w)\geq \frac{1}{m\ell}$. Therefore,   $c_0(f)\geq c_0(f_w)\geq \frac{1}{m\ell}.$

\bigskip

\textbf{Step C.}  Suppose that  $c>m\ell_m$.      
Since $w((x^{\alpha}+y^{\beta})^c)\leq w(x^{m\ell_m}y^{m\ell_m})$ by Claim~1,  if $\alpha, \beta \geq 2$, 
then we  immediately obtain $c\leq m\ell_m$. 
Therefore, either $\alpha=1$ or $\beta=1$. By exchanging coordinates if necessary, we may assume that $\alpha=1$.
Note that $\alpha=1$ implies that $w(x)\geq w(y)$ and $\frac{w(x)}{w(y)}$ is the integral number $\beta$.

Therefore, 
$$f_w=\epsilon x^ay^b(x+A_1y^\beta)^{c} (x^{\alpha_2}+g_2(x,y))^{c_2} \cdots(x^{\alpha_r}+g_r(x,y))^{c_r},$$ 
where $A_1$ is a non-zero constant.
The weighted leading term $f_w$ contains the monomial $x^{(a+c+\sum_{i=2}^r \alpha_ic_i)}y^b$  and  
$a+c+\sum_{i=2}^r \alpha_ic_i\geq c>m\ell_m $. 

We now apply a change of coordinate $x+A_1y^\beta\mapsto x$ to the polynomials $f_i$ and  $f$.  Set
$$ f_i^{(1)}(x,y):=f_i(x-A_1y^\beta,y) $$ 
for each $i$ and 
$$ f^{(1)}(x,y):=f(x-A_1y^\beta,y). $$ 
Then $f^{(1)}_w(x,y)=f_w(x-A_1y^\beta,y)$ and $f^{(1)}=\prod  f^{(1)}_i$.

Since  $f_1^{(1)},\ldots, f_{\ell_m}^{(1)} $ form a basis for the $(3m)$-jet space $k[x,y]/ (x, y)^{3m+1}$, we may again assume that $f_i^{(1)}$ contains the monomial $\mathbf{x}_i$ for each $i$. The Newton polygon of the polynomial $f^{(1)}$  again contains the point corresponding to~$x^{m\ell_m}y^{m\ell_m}$.

Now we go back to Step~A with $f^{(1)}$ instead of $f$, i.e., let $\Lambda^{(1)}$ be the edge of the Newton polygon of $f^{(1)}$ that intersects the line $s=t$. We also find weights $w^{(1)}(x), w^{(1)}(y)$ with respect to which the monomials of $f^{(1)}_{w^{(1)}}$ are plotted  on the edge $\Lambda^{(1)}$.  
The slope of the  edge $\Lambda^{(1)}$ is~$-\frac{w^{(1)}(x)}{w^{(1)}(y)}$. 

Let $L$ be the line in $\mathbb{R}^2$ determined by the edge $\Lambda$. We observe that 
\begin{quote}there is no monomial in $f^{(1)}$ plotted under the line $L$ and  that there is no monomial in $f^{(1)}$ plotted on the line $L$  with $s<c$. \ \ \ (*)\end{quote}
But $f^{(1)}$ possesses the monomial $x^{a+c+\sum_{i=2}^r \alpha_ic_i}y^b$ that is plotted on $L$.
This shows that  $\Lambda^{(1)}$ is  strictly steeper than $\Lambda$,  i.e., $-\frac{w^{(1)}(x)}{w^{(1)}(y)}<-\frac{w(x)}{w(y)}$.

As before, we write 
$$f^{(1)}_{w^{(1)}}=\epsilon^{(1)} x^{a^{(1)}}y^{b^{(1)}}(x^{\alpha^{(1)}_1}+g_1^{(1)}(x,y))^{c_1^{(1)}} \cdots(x^{\alpha^{(1)}_{r^{(1)}}}+g_{r^{(1)}}^{(1)}(x,y))^{c^{(1)}_{r^{(1)}}},$$ 
where $\epsilon^{(1)}$ is a non-zero constant and $g_i^{(1)}(x,y)$ is an irreducible  weighted homogeneous polynomial of degree $w^{(1)}(x^{\alpha_i^{(1)}})$ such that it does not contain the monomial $x^{\alpha_i^{(1)}}$. 

 Let $c^{(1)}=\max\left\{c_i^{(1)}\right\}.$
 Again we assume that $c^{(1)}=c_1^{(1)}$. The irreducible polynomial $g_1^{(1)}$ must 
  contain  the monomial $y^{\beta^{(1)}}$ for some positive integer $\beta^{(1)}$.

If $c^{(1)}\leq m\ell_m$, then the proof  is done by Claim~3. If $c^{(1)}> m\ell_m$, then we follow Step~C.  We here remark that  $\alpha^{(1)}$ must be $1$ 
because 
$\frac{w^{(1)}(x)}{w^{(1)}(y)}>\frac{w(x)}{w(y)}\geq 1$.  Note $\frac{w^{(1)}(x)}{w^{(1)}(y)}$ is the integral number $\beta^{(1)}$ and 
$\frac{w^{(1)}(x)}{w^{(1)}(y)}-\frac{w(x)}{w(y)}=\beta^{(1)}-\beta\geq 1$. Now we go back to Step~A with the new coordinate-changed polynomials  $f_i^{(2)}$ and  $f^{(2)}$. 

For the proof, it is enough to show that this procedure terminates in a finite number of loops.
The slope of $\Lambda^{(i)}$ is bounded below by $-m\ell_m$ since the Newton polygon of  $f^{(i)}$ must contain the point corresponding to~$x^{m\ell_m}y^{m\ell_m}$ and it has the property (*) above. The termination is therefore guaranteed because the slope of $\Lambda^{(i)}$ drops by at least $1$ for each loop.
\end{proof}

\section{K-stable del Pezzo surfaces}\label{section:K-stable del Pezzo surfaces}

In this section, we prove the first statement of Main Theorem. Before we start, we should remark here that the estimations of the $\delta$-invariants  in this section are not sharp at all. Since we have only to check that the $\delta$-invariants are strictly bigger than~$1$, our estimations have been made in such a way that we can reduce the amount of computation as much as possible.

Let $S_d$ be a smooth del Pezzo surface of degree $d\leq 5$.  Let $m$ be a positive integer and
let $s_1,\ldots, s_{\ell_m}$ be  sections in $\mathrm{H}^0(S_d,\mathcal{O}_{S_d}(-mK_{S_d}))$ that form a basis for $\mathrm{H}^0(S_d,\mathcal{O}_{S_d}(-mK_{S_d}))$, where $\ell_m=\frac{dm(m+1)}{2}+1$. Denote the divisor of the section $s_i$ by $D^m_i$. Set
$$D^m:=\sum_{i=1}^{\ell_m}D^m_i.$$

Since $d\leq 5$, for a given point $p\in S_d$, we can choose mutually disjoint $(9-d)$ $(-1)$-curves $M_1,\ldots,M_{9-d}$ which do not pass through the point $p$. By contracting these $(-1)$-curves, we obtain a birational morphism $\pi : S_d\rightarrow \mathbb{P}^2$ such that it is an isomorphism in a neighborhood of~$p$.  Set $q=\pi(p)$. By a suitable coordinate change, we may assume that $q=[0:0:1]$. Note that $c_p(S_d, D^m_i)=c_q(\mathbb{P}^2, \pi(D^m_i))$ for $i=1,\ldots, \ell_m$  and $c_p(S_d, D^m)=c_q(\mathbb{P}^2, \pi(D^m)).$  Denote~$\pi(D^m_i)$ by $\bar{D}^m_i$ for each $i$ and $\pi(D^m)$ by $\bar{D}^m$.

For an effective divisor $C$ in $ |-mK_{S_d}|$,  $\pi(C)$ is an effective divisor of degree $3m$ on $\mathbb{P}^2$ which passes through the points   $\pi(M_1),\ldots,\pi(M_{9-d})$ with multiplicities at least  $m$.  Such divisors produce an $\ell_m$-dimensional subspace of $\mathrm{H}^0(\mathbb{P}^2, \mathcal{O}_{\mathbb{P}^2}(-mK_{\mathbb{P}^2}))$.  We denote this subspace by $\mathcal{L}_m$. The effective divisors $\bar{D}^m_1,\ldots, \bar{D}^m_{\ell_m}$ induce a basis for $\mathcal{L}_m$.

As before, let $[x:y:z]$ be a homogeneous coordinate for $\mathbb{P}^2$.  We may consider $x$ and~$y$ as local coordinates in the affine chart  $U$ by $z\ne 0$.  For a polynomial $g(x,y)=g_i(x,y)+g_{i+1}(x,y)+\ldots$, where $g_j(x,y)$ is a homogeneous polynomial of degree $j$, we will call the lowest degree term~$g_i(x,y)$ the Zariski tangent term of $g(x,y)$.
 In $U$, the divisor $\bar{D}^m_i$ is defined by a polynomial $f_{m,i}$. 
Note that
$$m\ell_m \cdot c_0\left(\prod_{i=1}^{\ell_m}f_{m,i}\right)=c_q\left(\mathbb{P}^2, \frac{1}{m\ell_m}\bar{D}^m\right)=c_p\left(S_d, \frac{1}{m\ell_m}D^m\right).$$

For each $0\leq n\leq 3m$, we define a subspace  of the $k$-vector space  of homogeneous polynomials of degree $n$ in the variables $x$ and $y$ as follows:
\[\mathcal{T}_{n,m}:=\left\{\ g(x,y)\in k[x,y]_n \ \left| \
 \aligned
& \mbox{there is an effective divisor $C$ in $|-mK_{S_d}|$ such that}\\
& \mbox{the divisor $\pi(C)$ is defined in $U$ by a polynomial}\\
& \mbox{whose Zariski tangent term is $g(x,y)$.}\\
\endaligned\right.\right\},\]
where  $k[x,y]_n$  denotes the space  of homogeneous polynomials of degree $n$ in $x, y$.
Note that the space $k[x,y]_n$ is spanned by the standard basis
$$\mathcal{S}_{n}=\left\{x^n,x^{n-1}y,\ldots,xy^{n-1},y^n \right\}.$$
Set
\[\mathcal{T}_m :=\bigoplus_{n=0}^{3m}\mathcal{T}_{n,m}.\]
A basis for $\mathcal{T}_m$ consisting only of homogeneous elements will be called a basis  for $\mathcal{T}_m$.

Put $\ell= \dim_k(\mathcal{T}_m)$.
Let $\{t_1,\ldots, t_{\ell} \}$ be a basis of $\mathcal{T}_m$. For each $1\leq j\leq \ell$ we  take a polynomial $h_{m,j}$ on $U$ originated from $\mathcal{L}_m$ such that its Zariski tangent term is $t_j$.  
It is obvious that $h_{m,1},\ldots, h_{m,\ell}$ are linearly independent, so that    $\ell\leq \ell_m$.  Let $G$ be a member in $\mathcal{L}_m$. It can be expressed on $U$ by  a polynomial  $g(x,y)=g_i(x,y)+g_{i+1}(x,y)+\ldots$, where $g_j(x,y)$ is a homogeneous polynomial of degree $j$. If $i=3m$, then $g(x,y)$ is a linear combination of $h_{m,1},\ldots, h_{m,\ell}$.
Assume that if $a+1\leq i\leq 3m$, then $g(x,y)$ is a linear combination of $h_{m,1},\ldots, h_{m,\ell}$. For $i=a$, we have constants $\alpha_1,\ldots, \alpha_\ell$ such that the Zariski tangent term of
$\sum \alpha_j h_j(x,y)$ is  $g_i(x,y)$. Since $g(x,y)-\sum\alpha_jh_{m,j}(x,y)$ is induced by some element in  $\mathcal{L}_m$, the polynomial $g(x,y)$ is a linear combination of $h_{m,1},\ldots, h_{m,\ell}$. Thus inductively we can conclude  that every polynomial on $U$ from $\mathcal{L}_m$ is a linear combination of $h_{m,1},\ldots, h_{m,\ell}$.
Consequently, the dimension of $\mathcal{T}_m$ is equal to the dimension of $\mathrm{H}^0(S_d,\mathcal{O}_{S_d}(-mK_{S_d}))$, i.e., $\ell=\ell_m$.

Now let us explain how to estimate the log canonical threshold 
$$c_0\left(\prod_{i=1}^{\ell_m}f_{m,i}\right).$$
\bigskip

\textbf{Step 0.}  Consider the space $\mathcal{L}_1$. Using a suitable coordinate change, we choose a basis~$\{t_1,\ldots, t_{d+1}\}$   for the space $\mathcal{T}_1$ with $\deg(t_1)\leq\ldots\leq \deg(t_{d+1})$ such that $t_1,\ldots, t_{d}$ are monomials.  Indeed, in each proof, we will immediately see that this is  possible whenever we need this step. 
Denote the monomial $t_i$ by $\mathbf{x}_{1,i}$ for $i=1,\ldots, d$. Choose an appropriate monomial from those in  $t_{d+1}$ and denote it by $\mathbf{x}_{1,d+1}$.
\medskip

\textbf{Step 1.} 
Consider the finite sets
\[\mathcal{C}_m:=\left\{ \prod_{i=1}^{d+1} \mathbf{x}_{1,i}^{n_i} \ \Big| \ \mbox{$n_i$ are non-negative integers with $n_1+\cdots+n_{d+1}=m$}\right\},\]
\[\mathcal{B}^s_m:=\left\{ \prod_{i=1}^{d+1} t_i^{n_i} \ \Big| \ \mbox{$n_i$ are non-negative integers with $n_1+\cdots+n_{d+1}=m$}\right\}.\]
The monomial $\prod_{i=1}^{d+1} \mathbf{x}_{1,i}^{n_i}$ is contained in the homogeneous polynomial $\prod_{i=1}^{d+1} t_i^{n_i}$.
Since the sets $\{t_1,\ldots, t_{d+1}\}$ and $\{\mathbf{x}_{1,1},\ldots, \mathbf{x}_{1,d+1}\}$ are linearly independent respectively,
the sets $\mathcal{C}_m$ and $\mathcal{B}^s_m$ are linearly independent respectively.

Denote the number of elements in $\mathcal{B}^s_m$ by $b$. Note that $b= \binom{m+d}{d}\leq \ell_m$. Set $a=\ell_m-b$.
\medskip

 \textbf{Step 2.} 
 We choose linearly independent $a$ homogeneous elements $u_1,\ldots, u_a$ in $\mathcal{T}_m$ such that
  the set $\{u_1,\ldots, u_a\}\cup\mathcal{B}^s_m$ forms a basis for $\mathcal{T}_m$.
 \begin{lemma}\label{lemma:injection1}
There is an injective map 
 \begin{equation}\label{index-map2}\iota_m: \{u_1,\ldots, u_a\}\cup\mathcal{B}^s_m\to \bigcup_{n=0}^{3m}\mathcal{S}_n\end{equation}
 such that
 \begin{enumerate}
  \item $\iota_m(\prod_{i=1}^{d+1} t_i^{n_i})=\prod_{i=1}^{d+1} \mathbf{x}_{1,i}^{n_i}$;
  \item $\iota_m(u_i)$ is contained in the homogeneous polynomial $u_i$ for each $i$.
  \end{enumerate}
 \end{lemma}
 \begin{proof}
See Appendix at the end.
\end{proof}
\medskip

 \textbf{Step 3.} Consider all the possibilities of  the image $\iota_m(\{u_1,\ldots, u_a\})$  with linearly independent~$a$ homogeneous elements $u_1,\ldots, u_a$ in $\mathcal{T}_m$ such that
  the set $\{u_1,\ldots, u_a\}\cup\mathcal{B}^s_m$ forms a basis for $\mathcal{T}_m$.

\begin{lemma}\label{lemma:injection2}
There is  an injective map 
\begin{equation}\label{index-map0}I_m :\{ f_{m, i}\ | \ 1\leq i\leq \ell_m\}\to \bigcup_{n=0}^{3m}\mathcal{S}_n\end{equation}
such that 
\begin{enumerate}
\item the monomial $I_m (f_{m,i})$ is contained in $f_{m,i}$;
\item its image  coincides with the image of $\iota_m$.
\end{enumerate}
\end{lemma}
\begin{proof}
See Appendix at the end.
\end{proof}

 \medskip
 
Lemma~\ref{lemma:injection2} implies that the Newton polygon of the polynomial $\prod_{i=1}^{\ell_m} f_{m,i}$ contains the point corresponding to the monomial
\begin{equation}\label{product} x^{\alpha}y^{\beta}:=\prod_{i=1}^{\ell_m}I_m(f_{m,i})=\prod_{i=1}^a\iota_m(u_i)\cdot \prod_{\mathbf{x}\in\mathcal{C}_m}\mathbf{x}.\end{equation}
Varying the possible image $\iota_m(\{u_1,\ldots, u_a\})$, we find an attainable maximum $v_m$ of the value $\max\{\alpha, \beta\}$. Then we see that the Newton polygon of $\prod_{i=1}^{\ell_m} f_{m,i}$ always contains the point corresponding to the monomial $x^{v_m}y^{v_m}$.
\medskip

\textbf{Step 4.} With the monomial $x^{v_m}y^{v_m}$, we can follow Steps~A, B, C in the proof of Theorem~\ref{theorem:d=9}.
Then we obtain $$c_0(\prod_{i=1}^{\ell_m} f_{m, i})\geq \frac{1}{v_m}.$$ 
This implies $$m\ell_m \cdot c_0\left(\prod_{i=1}^{\ell_m}f_{m,i}\right)=c_q\left(\mathbb{P}^2, \frac{1}{m\ell_m}\bar{D}^m\right)=c_p\left(S_d, \frac{1}{m\ell_m}D^m\right)\geq \frac{m\ell_m}{v_m}$$
for an arbitrary point $p$ on $S_d$.
Therefore,
$$\delta (S_d)\geq \limsup_m \frac{m\ell_m}{v_m}.$$
\medskip

In Step~C we define the polynomials $f_i^{(1)}$ by applying a change of coordinate $x+A_1y^\beta\mapsto x$ to the polynomial $f_i$ for each $i$. We furthermore define $f_i^{(k)}$ inductively for each $i$ and each $k$. We here apply the same to $f_{m,i}$  and denote the coordinate-changed polynomials by $f_{m,i}^{(k)}$. 
In  the change of coordinate $x+A_1y^\beta\mapsto x$, we may assume that~$\beta>1$. Indeed, if $\beta=1$, then we replace $f_{m, i}(x,y)$ by $f_{m,i}(x-A_1y, x)$. Then they also form a basis for $\mathcal{L}_m$ and the log canonical threshold of their product at the origin is equal to that of the product of the original $f_{m, i}$. Note that the edge of the Newton polygon of the product of the replaced $f_{m, i}$ that intersects the line $s=t$ is not parallel to the line $s=-t$. This implies that $\beta$ for the replaced $f_{m, i}$ cannot be $1$ in Step~C. Therefore, it is strictly bigger than $1$ and hence $\beta^{(k)}$ for $f^{(k)}_m:=\prod_{i=1}^{\ell_m}f_{m,i}^{(k)}$ is also strictly bigger than $1$ for all $k$.
The lemma below guarantees that we are able to go through  Step~C in Step~4.

\begin{lemma}\label{lemma:injection-after-coordinate-change}
For each $k$ there is  an injective map 
\[I_{m,k}:\{ f_{m, i}^{(k)}\ | \ 1\leq i\leq \ell_m\}\to \bigcup_{n=0}^{3m}\mathcal{S}_n\]
such that 
\begin{enumerate}
\item the monomial $I_{m,k} (f_{m,i}^{(k)})$ is contained in $f_{m,i}^{(k)}$;
\item its image  coincides with the image of $\iota_m$.
\end{enumerate}
\end{lemma}
\begin{proof}
See Appendix at the end.
\end{proof}


Note that in some cases, we may skip Steps~0 and~1, i.e., we may start from Step~2 with~$a=\ell_m$.
\smallskip

As we see, the $\delta$-invariant is defined in an asymptotic way. Therefore, it is enough to see the asymptotic behaviors of $m\ell_m$ and $v_m$ with respect to $m$. 
Since $m\ell_m =\frac{d}{2}m^3+\frac{d}{2}m^2+m$, a lower bound of the $\delta$-invariant of $S_d$ can be determined by 
the coefficient of the term $m^3$  in $v_m$. For reader's convenience, we here provide the following diagrams. In each diagram, 
the exponent of $x$ in the product of all the monomials in $x$, $y$ corresponding to the integral points of a triangle (or a quadrilateral) is given by a cubic polynomial with respect to $m$. The number in the triangle (or the quadrilateral) is the coefficient of the term $m^3$ in this cubic polynomial. The diagrams in Figure~\ref{table} enable us to immediately figure out  the coefficients of the term $m^3$  in the exponents $v_m$  that appear in the present and the next sections.

\medskip
\begin{figure}[H]
\caption{}\label{table}
\begin{minipage}[m]{.4\linewidth}
\begin{center}
\begin{tikzpicture}[scale=0.77, line cap=round,line join=round,>=triangle 45,x=0.2cm,y=0.2cm]
\clip(-5.,-35.) rectangle (38.,8.);
\draw [->] (0.,-30.) -- (0.,5.);
\draw [->] (0.,-30.) -- (35.,-30.);
\draw (0.,0.)-- (30.,-30.);
\draw (9.8,-30) node[anchor=north west] {$m$};
\draw (-4,-18) node[anchor=north west] {$m$};
\draw (19.6,-30) node[anchor=north west] {$2m$};
\draw (-5,-8) node[anchor=north west] {$2m$};
\draw (29.7,-30) node[anchor=north west] {$3m$};
\draw (-5,2) node[anchor=north west] {$3m$};
\draw (35,-29) node[anchor=north west] {$x$};
\draw (-1.3,8) node[anchor=north west] {$y$};
\draw (0.,-20.)-- (20.,-20.);
\draw (20.,-30.)-- (20.,-20.);
\draw (10.,-30.)-- (10.,-20.);
\draw (10.,-10.)-- (10.,-20.);
\draw (0.,-10.)-- (10.,-10.);
\draw (0.,-30.)-- (10.,-20.);
\draw (10.,-30.)-- (0.,-20.);
\draw (20.,-20.)-- (10.,-30.);
\draw (20.,-30.)-- (0.,-10.);
\draw (10.,-10.)-- (0.,-20.);
\draw (10.,-10.)-- (20.,-10.);
\draw (20.,-10.)-- (20.,-20.);
\draw (3.0,-19.5) node[anchor=north west] {$\frac{3}{24}$};
\draw (-0.5,-23) node[anchor=north west] {$\frac{1}{24}$};
\draw (3.0,-25.5) node[anchor=north west] {$\frac{3}{24}$};
\draw (6.1,-23) node[anchor=north west] {$\frac{5}{24}$};
\draw (9.5,-23) node[anchor=north west] {$\frac{7}{24}$};
\draw (16.3,-23) node[anchor=north west] {$\frac{11}{24}$};
\draw (13.1,-19.5) node[anchor=north west] {$\frac{9}{24}$};
\draw (13.1,-25.5) node[anchor=north west] {$\frac{9}{24}$};
\draw (-0.5,-13) node[anchor=north west] {$\frac{1}{24}$};
\draw (3.0,-15.5) node[anchor=north west] {$\frac{3}{24}$};
\draw (3.0,-9.5) node[anchor=north west] {$\frac{3}{24}$};
\draw (6.1,-13) node[anchor=north west] {$\frac{5}{24}$};
\draw (1.4,-5) node[anchor=north west] {$\frac{4}{24}$};
\draw (11.5,-15) node[anchor=north west] {$\frac{16}{24}$};
\draw (20.5,-25) node[anchor=north west] {$\frac{28}{24}$};
\draw (15,-10.5) node[anchor=north west] {$\frac{20}{24}$};
\end{tikzpicture}
\end{center}
\end{minipage}
\hfill
\begin{minipage}[m]{.4\linewidth}
\begin{center}
\begin{tikzpicture}[scale=0.77, line cap=round,line join=round,>=triangle 45,x=0.2cm,y=0.2cm]
\clip(-5.,-35.) rectangle (38.,8.);
\draw [->] (0.,-30.) -- (0.,5.);
\draw [->] (0.,-30.) -- (35.,-30.);
\draw (0.,0.)-- (30.,-30.);
\draw (-4,-18) node[anchor=north west] {$m$};
\draw (19.6,-30) node[anchor=north west] {$2m$};
\draw (-5,-8) node[anchor=north west] {$2m$};
\draw (29.7,-30) node[anchor=north west] {$3m$};
\draw (-5,2) node[anchor=north west] {$3m$};
\draw (35,-29) node[anchor=north west] {$x$};
\draw (-1.3,8) node[anchor=north west] {$y$};
\draw (24,-30) node[anchor=north west] {$\frac{5}{2}m$};
\draw [color=sqsqsq] (25.,-30.)-- (20.,-30.);
\draw [color=sqsqsq] (20.,-30.)-- (12.398655527744687,-22.603272535972017);
\draw [color=sqsqsq] (12.398655527744687,-22.603272535972017)-- (12.482897438943981,-17.548757864014473);
\draw [color=sqsqsq] (12.482897438943981,-17.548757864014473)-- (25.,-30.);
\draw [dash pattern=on 1pt off 1pt] (12.398655527744687,-22.603272535972025)-- (12.482897438943981,-30.);
\draw (12,-30) node[anchor=north west] {$\frac{5}{4}m$};
\draw (16.5,-24) node[anchor=north west] {$1$};
\end{tikzpicture}
\end{center}

\end{minipage}

\begin{minipage}[m]{.4\linewidth}
\begin{center}
\begin{tikzpicture}[scale=0.77, line cap=round,line join=round,>=triangle 45,x=0.2cm,y=0.2cm]
\clip(-5.,-35.) rectangle (38.,8.);
\draw [->] (0.,-30.) -- (0.,5.);
\draw [->] (0.,-30.) -- (35.,-30.);
\draw (0.,0.)-- (30.,-30.);
\draw (9.8,-30) node[anchor=north west] {$m$};
\draw (-4,-18) node[anchor=north west] {$m$};
\draw (14,-30) node[anchor=north west] {$\frac{3}{2}m$};
\draw (19.6,-30) node[anchor=north west] {$2m$};
\draw (-5,-8) node[anchor=north west] {$2m$};
\draw (29.7,-30) node[anchor=north west] {$3m$};
\draw (-5,2) node[anchor=north west] {$3m$};
\draw (35,-29) node[anchor=north west] {$x$};
\draw (-1.3,8) node[anchor=north west] {$y$};
\draw (10.,-10.)-- (10.,-30.);
\draw (10.,-30.)-- (20.,-30.);
\draw (20.,-30.)-- (10.,-10.);
\draw (15.,-30.)-- (10.,-20.);
\draw (10,-25) node[anchor=north west] {$\frac{7}{24}$};
\draw (12.486142017771918,-22.174404624411963) node[anchor=north west] {$\frac{25}{24}$};
\end{tikzpicture}

\end{center}
\end{minipage}
\hfill
\begin{minipage}[m]{.4\linewidth}
\begin{center}
\begin{tikzpicture}[scale=0.77, line cap=round,line join=round,>=triangle 45,x=0.2cm,y=0.2cm]

\clip(-5.,-35.) rectangle (38.,8.);
\draw [->] (0.,-30.) -- (0.,5.);
\draw [->] (0.,-30.) -- (35.,-30.);
\draw (0.,0.)-- (30.,-30.);
\draw (9.8,-30) node[anchor=north west] {$m$};
\draw (-4,-18) node[anchor=north west] {$m$};
\draw (-5,-13) node[anchor=north west] {$\frac{3}{2}m$};
\draw (19.6,-30) node[anchor=north west] {$2m$};
\draw (-5,-8) node[anchor=north west] {$2m$};
\draw (29.7,-30) node[anchor=north west] {$3m$};
\draw (-5,2) node[anchor=north west] {$3m$};
\draw (35,-29) node[anchor=north west] {$x$};
\draw (-1.3,8) node[anchor=north west] {$y$};
\draw (0.,-20.)-- (20.,-20.);
\draw (0.,-10.)-- (0.,-20.);
\draw (0.,-10.)-- (20.,-20.);
\draw (0.,-15.)-- (10.,-20.);
\draw (0.3,-15.7) node[anchor=north west] {$\frac{2}{24}$};
\draw (5,-13.7) node[anchor=north west] {$\frac{14}{24}$};
\draw (35,-29) node[anchor=north west] {$x$};
\draw (-1.3,8) node[anchor=north west] {$y$};

\end{tikzpicture}

\end{center}
\end{minipage}

\end{figure}
\medskip


Now we are ready to estimate the values of the $\delta$-invariants of smooth del Pezzo surfaces.

We keep the same notations as before.

 \begin{theorem}\label{theorem:d1}
Let $S$ be a del Pezzo surface of degree $1$. Then $\delta(S)\geq \frac{3}{2}$.
\end{theorem}

\begin{proof} 
Let $C$ be a cubic curve in $\mathbb{P}^2$ that  passes through the points $\pi(M_1),\ldots,\pi(M_{8})$ and $q$. Since $\pi(M_1),\ldots,\pi(M_{8})$ are in general position, the curve $C$ must be irreducible and reduced.

\medskip
\emph{Case 1.  The curve $C$ is smooth at $q$.}
\medskip

The effective divisor $G$ produced by a section in $\mathcal{L}_m$ can pass through the point $q$ with multiplicity at most $m$. To see this, we write $G=m'C+\Omega$, where $m'$ is a non-negative integer not bigger than $m$ and $\Omega$ is an effective divisor whose support does not contain $C$. If $\mathrm{mult}_q(G)>m$, then $\mathrm{mult}_q(\Omega)>m-m'$. This yields an absurd inequality 
\[9(m-m')=C\cdot\Omega\geq \mathrm{mult}_q(\Omega)+\sum_{i=1}^8\mathrm{mult}_{\pi(M_i)}(C)\cdot\mathrm{mult}_{\pi(M_i)}(\Omega)>9(m-m').\]
 Therefore, the image of the injective map in \eqref{index-map2} is always contained in the set $$\{x^{n_1}y^{n_2} ~|~ 0\leq n_1+n_2\leq m \}.$$
 These monomials are plotted in  the shade area of the following diagram:
 \medskip
 
 
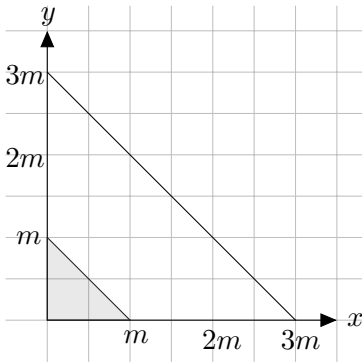
\begin{center} 
 \begin{tikzpicture}[scale=0.55, line cap=round,line join=round,>=triangle 45,x=0.2cm,y=0.2cm]
\draw [color=cqcqcq,, xstep=1.0cm,ystep=1.0cm] (-5.,-35.) grid (38.,8.);
\clip(-5.,-35.) rectangle (38.,8.);
\fill[color=sqsqsq,fill=sqsqsq,fill opacity=0.1] (0.,-30.) -- (0.,-20.) -- (10.,-30.) -- cycle;
\draw [->] (0.,-30.) -- (0.,5.);
\draw [->] (0.,-30.) -- (35.,-30.);
\draw (0.,0.)-- (30.,-30.);
\draw (8,-30) node[anchor=north west] {$m$};
\draw (-5,-18) node[anchor=north west] {$m$};
\draw (17.5,-30) node[anchor=north west] {$2m$};
\draw (-6.3,-8) node[anchor=north west] {$2m$};
\draw (27,-30) node[anchor=north west] {$3m$};
\draw (-6.3,2) node[anchor=north west] {$3m$};
\draw (35,-28) node[anchor=north west] {$x$};
\draw (-2,9) node[anchor=north west] {$y$};
\draw [color=sqsqsq] (0.,-30.)-- (0.,-20.);
\draw [color=sqsqsq] (0.,-20.)-- (10.,-30.);
\draw [color=sqsqsq] (10.,-30.)-- (0.,-30.);
\end{tikzpicture}
\end{center}
 \medskip
 Since $m\ell_m=\frac{1}{2}m(m^2+m+2)$, Figure~\ref{table} shows
 $$\limsup_m \frac{m\ell_m}{v_m}=3.$$

\medskip
\emph{Case 2.  The curve $C$ is singular at $q$.}
\medskip

Then $C$ is a unique curve in $\mathbb{P}^2$ that  passes through the points $\pi(M_1),\ldots,\pi(M_{8})$ and $q$. Note that $\dim_k\mathcal{T}_1=2.$  Using  a suitable coordinate change,
if $C$ has a node at $q$, then we may  assume that $$t_1=1, \ t_2=xy,$$ and  if $C$ has a cusp at $q$, then we may  assume that $$t_1=1, \ t_2=(x+y)^2.$$
Therefore, in both the cases we can always take $$\mathbf{x}_{1,1}=1, \ \mathbf{x}_{1,2}=xy.$$ 
We then obtain 
$$\mathcal{C}_m=\{1,xy,x^2y^2,\ldots,x^my^m\}. $$  
Let $G$ be the effective divisor of a section in $\mathcal{L}_m$.  We write $G=m'C+\Omega$ as in Case~1.
Since $9(m-m')=C\cdot\Omega\geq 8(m-m')+2\mult_q(\Omega)$, we obtain $$\mult_q(G)=m'\cdot \mult_q(C)+\mult_q(\Omega)\leq 2m'+\frac{m-m'}{2}.$$
This shows that we can define the injective map $\iota_m$ in \eqref{index-map2}  in such a way that its image is contained in the set
\[\bigcup_{i=0}^{m}\left\{x^{i+n_1}y^{i+n_2}~|~0\leq n_1, n_2 \mbox{ and } 0\leq n_1+n_2\leq \frac{m-i}{2} \right\}.\]
 These monomials are plotted in  the shade area of the following diagram:

 \medskip
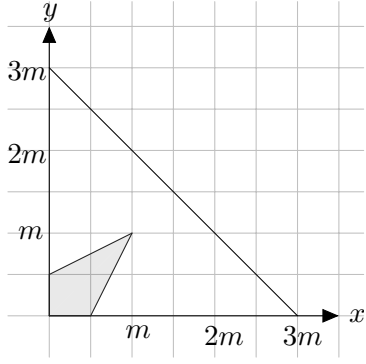
\begin{center}

\begin{tikzpicture}[scale=0.55, line cap=round,line join=round,>=triangle 45,x=0.2cm,y=0.2cm]
\draw [color=cqcqcq,, xstep=1.0cm,ystep=1.0cm] (-5.,-35.) grid (38.,8.);
\clip(-5.,-35.) rectangle (38.,8.);
\fill[color=sqsqsq,fill=sqsqsq,fill opacity=0.1] (0.,-30.) -- (10.,-20.) -- (5.,-30.) -- cycle;
\fill[color=sqsqsq,fill=sqsqsq,fill opacity=0.1] (0.,-30.) -- (10.,-20.) -- (0.,-25.) -- cycle;
\draw [->] (0.,-30.) -- (0.,5.);
\draw [->] (0.,-30.) -- (35.,-30.);
\draw (0.,0.)-- (30.,-30.);
\draw (8,-30) node[anchor=north west] {$m$};
\draw (-5,-18) node[anchor=north west] {$m$};
\draw (17.5,-30) node[anchor=north west] {$2m$};
\draw (-6.3,-8) node[anchor=north west] {$2m$};
\draw (27,-30) node[anchor=north west] {$3m$};
\draw (-6.3,2) node[anchor=north west] {$3m$};
\draw (35,-28) node[anchor=north west] {$x$};
\draw (-2,9) node[anchor=north west] {$y$};
\draw [color=sqsqsq] (0.,-25.)-- (10.,-20.);
\draw [color=sqsqsq] (10.,-20.)-- (5.,-30.);
\draw [color=sqsqsq] (5.,-30.)-- (0.,-30.);
\draw [color=sqsqsq] (0.,-25.)-- (0.,-30.);
\end{tikzpicture}

\end{center}
 \medskip
The exponent of $x$ in the product of all the monomials corresponding to the integral points of the shade area in the diagram above
is clearly smaller than the exponent from the shade area in the  below.
 \medskip

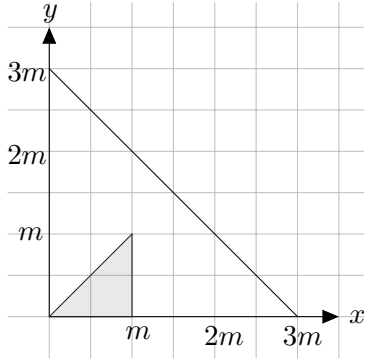
\begin{center}

\begin{tikzpicture}[scale=0.55, line cap=round,line join=round,>=triangle 45,x=0.2cm,y=0.2cm]
\draw [color=cqcqcq,, xstep=1.0cm,ystep=1.0cm] (-5.,-35.) grid (38.,8.);
\clip(-5.,-35.) rectangle (38.,8.);
\fill[color=sqsqsq,fill=sqsqsq,fill opacity=0.1] (0.,-30.) -- (10.,-20.) -- (10.,-30.) -- cycle;
\draw [->] (0.,-30.) -- (0.,5.);
\draw [->] (0.,-30.) -- (35.,-30.);
\draw (0.,0.)-- (30.,-30.);
\draw (8,-30) node[anchor=north west] {$m$};
\draw (-5,-18) node[anchor=north west] {$m$};
\draw (17.5,-30) node[anchor=north west] {$2m$};
\draw (-6.3,-8) node[anchor=north west] {$2m$};
\draw (27,-30) node[anchor=north west] {$3m$};
\draw (-6.3,2) node[anchor=north west] {$3m$};
\draw (35,-28) node[anchor=north west] {$x$};
\draw (-2,9) node[anchor=north west] {$y$};
\draw [color=sqsqsq] (0.,-30.)-- (10.,-20.);
\draw [color=sqsqsq] (10.,-20.)-- (10.,-30.);
\draw [color=sqsqsq] (10.,-30.)-- (0.,-30.);
\end{tikzpicture}

\end{center}
 \medskip
Therefore, Figure~\ref{table} shows 
$$\limsup_m \frac{m\ell_m}{v_m}\geq \frac{3}{2}.$$

Consequently,  from Case~1~and Case~2 we obtain $\delta(S)\geq \frac{3}{2}.$ \end{proof}

 \begin{theorem}\label{theorem:d2}
Let $S$ be a del Pezzo surface of degree $2$. Then $\delta(S)\geq \frac{6}{5}$.
\end{theorem}

\begin{proof} 
Note that $\dim_k(\mathcal{T}_1)=3$.  We have two possibilities: either $\deg (t_1)=0$, $\deg (t_2)=\deg (t_3)=1$
or $\deg (t_1)=0$, $\deg (t_2)=1$, $\deg (t_3)=2$.
 
\medskip
\emph{Case 1.  $\deg (t_1)=0$, $\deg (t_2)=\deg (t_3)=1$.}
\medskip

In this case, we may assume that $$t_1=1, \ t_2=x,\ t_3=y$$ by a suitable coordinate change. Therefore, $\mathbf{x}_{1,1}=1$, $\mathbf{x}_{1,2}=x$  and 
$\mathbf{x}_{1,3}=y$.
We then obtain $$\mathcal{C}_m=\{1,x,y,x^2,xy,y^2,\ldots,x^m,x^{m-1}y,\ldots, y^m\}. $$  

We claim that the effective divisor $G$ yielded by a section in $\mathcal{L}_m$  cannot pass through  $q$ with multiplicity more than $2m$. For the claim, we consider the effective  plurianticanonical divisor $$\tilde{G}=\pi^*(G)-m\sum_{i=1}^7 M_i$$ on $S$. 
Since $2m=H\cdot \tilde{G}\geq \mathrm{mult}_p(\tilde{G})$ for a general member $H$ in $|-K_S|$ passing through the point $p$, we obtain $\mathrm{mult}_p(\tilde{G})\leq 2m$. This implies the claim.

Therefore, the image of $\{u_1,\ldots, u_a\}$ under the injective map $\iota_m$ in \eqref{index-map2} is always contained in
$$\bigcup_{n=0}^{2m}\mathcal{S}_n\setminus\mathcal{C}_m,$$
where $a=\ell_m-\frac{(m+1)(m+2)}{2}=\frac{m(m-1)}{2}$.
In case when 
\[\iota_m (\{u_1,\ldots, u_a\})=\{x^{n_1}y^{n_2}~|~0\leq n_1, n_2, m+2\leq n_1+n_2\leq 2m  \text{ and } m+2\leq n_1 \leq 2m \}\]
the value $\alpha$ in the monomial $x^{\alpha}y^{\beta}$ of \eqref{product} attains the possible maximum $v_m$. 
Therefore, the value $v_m$ can be asymptotically evaluated by the shade area in the following diagram:

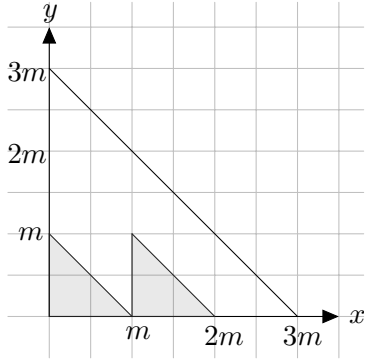
\begin{center}
\begin{tikzpicture}[scale=0.55, line cap=round,line join=round,>=triangle 45,x=0.2cm,y=0.2cm]
\draw [color=cqcqcq,, xstep=1.0cm,ystep=1.0cm] (-5.,-35.) grid (38.,8.);
\clip(-5.,-35.) rectangle (38.,8.);
\fill[color=sqsqsq,fill=sqsqsq,fill opacity=0.1] (0.,-20.) -- (0.,-30.) -- (10.,-30.) -- cycle;
\fill[color=sqsqsq,fill=sqsqsq,fill opacity=0.1] (10.,-30.) -- (10.,-20.) -- (20.,-30.) -- cycle;
\draw [->] (0.,-30.) -- (0.,5.);
\draw [->] (0.,-30.) -- (35.,-30.);
\draw (0.,0.)-- (30.,-30.);
\draw (8,-30) node[anchor=north west] {$m$};
\draw (-5,-18) node[anchor=north west] {$m$};
\draw (17.5,-30) node[anchor=north west] {$2m$};
\draw (-6.3,-8) node[anchor=north west] {$2m$};
\draw (27,-30) node[anchor=north west] {$3m$};
\draw (-6.3,2) node[anchor=north west] {$3m$};
\draw (35,-28) node[anchor=north west] {$x$};
\draw (-2,9) node[anchor=north west] {$y$};
\draw [color=sqsqsq] (0.,-20.)-- (0.,-30.);
\draw [color=sqsqsq] (0.,-30.)-- (10.,-30.);
\draw [color=sqsqsq] (10.,-30.)-- (0.,-20.);
\draw [color=sqsqsq] (10.,-30.)-- (10.,-20.);
\draw [color=sqsqsq] (10.,-20.)-- (20.,-30.);
\draw [color=sqsqsq] (20.,-30.)-- (10.,-30.);
\end{tikzpicture}

\end{center}
i.e., Figure~\ref{table} shows 
$$\limsup_m \frac{m\ell_m}{v_m}= \frac{6}{5}.$$

\medskip
\emph{Case 2.   $\deg (t_1)=0$, $\deg (t_2)=1$, $\deg (t_3)=2$.}
\medskip

In this case, there is a unique cubic  $C$  on $\mathbb{P}^2$ that passes through the points $\pi(M_1),\ldots,\pi(M_{7})$ and that has  multiplicity $2$ at the point~$q$.

Let $G$ be the effective divisor of a section in $\mathcal{L}_m$.

\medskip
\emph{Subcase 1. The curve $C$ is irreducible.}
\medskip

In this subcase, we may  assume that the Zariski tangent term of the defining polynomial of~$C$ on $U$ contains the monomial $xy$.

  We write $G=m'C+\Omega$, where $m'$ is a non-negative integer not bigger than $m$ and $\Omega$ is an effective divisor whose support does not contain $C$. Since
$$9(m-m')=C\cdot\Omega\geq 7(m-m')+2\mult_q(\Omega),$$
 the injective map $\iota_m$ in \eqref{index-map2} can be defined in a way that its image is  contained in the set
\[\bigcup_{i=0}^{m}\left\{x^{i+n_1}y^{i+n_2}~|~0\leq n_1, n_2 \mbox{ and } 0\leq n_1+n_2\leq m-i \right\}.\]
This set can be depicted as follows:

 \medskip

\begin{center}
\begin{tikzpicture}[scale=0.55, line cap=round,line join=round,>=triangle 45,x=0.2cm,y=0.2cm]
\draw [color=cqcqcq,, xstep=1.0cm,ystep=1.0cm] (-5.,-35.) grid (38.,8.);
\clip(-5.,-35.) rectangle (38.,8.);
\fill[color=sqsqsq,fill=sqsqsq,fill opacity=0.1] (0.,-20.) -- (10.,-20.) -- (10.,-30.) -- (0.,-30.) -- cycle;
\draw [->] (0.,-30.) -- (0.,5.);
\draw [->] (0.,-30.) -- (35.,-30.);
\draw (0.,0.)-- (30.,-30.);
\draw (8,-30) node[anchor=north west] {$m$};
\draw (-5,-18) node[anchor=north west] {$m$};
\draw (17.5,-30) node[anchor=north west] {$2m$};
\draw (-6.3,-8) node[anchor=north west] {$2m$};
\draw (27,-30) node[anchor=north west] {$3m$};
\draw (-6.3,2) node[anchor=north west] {$3m$};
\draw (35,-28) node[anchor=north west] {$x$};
\draw (-2,9) node[anchor=north west] {$y$};
\draw [color=sqsqsq] (0.,-20.)-- (10.,-20.);
\draw [color=sqsqsq] (10.,-20.)-- (10.,-30.);
\draw [color=sqsqsq] (10.,-30.)-- (0.,-30.);
\draw [color=sqsqsq] (0.,-30.)-- (0.,-20.);
\end{tikzpicture}

\end{center}
 \medskip
Figure~\ref{table} shows 
$$\limsup_m \frac{m\ell_m}{v_m}= 2.$$

\medskip
\emph{Subcase 2. The curve $C$ is reducible. }
\medskip

The cubic $C$ consists of a line $L$ and an irreducible conic $Q$. Since $(-K_S)^2=2$, it cannot consist of three lines. Note that $L$ passes through exactly two points of $\pi(M_1),\ldots,\pi(M_{7})$ and that  $Q$ passes through five of them. 
We may  assume that the Zariski tangent term of the defining polynomial of~$L$ on $U$ contains the monomial $x$ and that of $Q$ contains the monomial $y$.

We write $G=m_1L+m_2Q+\Omega$, where $m_1$ and $m_2$ are non-negative integers and $\Omega$ is an effective divisor whose support contains neither $L$ nor $Q$.

From the inequalities $$3m=L\cdot(m_1L+m_2Q+\Omega)\geq m_1+2m_2+2(m-m_1)+\mult_q\Omega,$$ 
$$6m=Q\cdot(m_1L+m_2Q+\Omega)\geq 2m_1+4m_2+5(m-m_2)+\mult_q\Omega,$$
we obtain $-m_1+2m_2\leq m$ and $2m_1-m_2\leq m$. These imply  $m_1, m_2\leq m$.
Moreover, we obtain 
\[\mult_q (\Omega)\leq m+\min\{-2m_1+m_2, m_1-2m_2\}.\] 
Therefore, we can define the injective map $\iota_m$ of \eqref{index-map2} in a way that its image is  contained in the sets either
\[\bigcup_{0\leq m_1\leq m_2\leq m}\left\{x^{m_1+n_1}y^{m_2+n_2}~|~0\leq n_1, n_2 \mbox{ and } 0\leq n_1+n_2\leq m+m_1-2m_2 \right\}
\]
or
\[\bigcup_{0\leq m_2\leq m_1\leq m}\left\{x^{m_1+n_1}y^{m_2+n_2}~|~0\leq n_1, n_2 \mbox{ and } 0\leq n_1+n_2\leq m-2m_1+m_2 \right\}
.\] 
Both the sets sit in the shade area of the following diagram:
\medskip

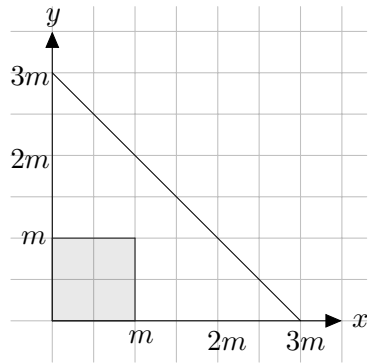
\begin{center}
\begin{tikzpicture}[scale=0.55, line cap=round,line join=round,>=triangle 45,x=0.2cm,y=0.2cm]
\draw [color=cqcqcq,, xstep=1.0cm,ystep=1.0cm] (-5.,-35.) grid (38.,8.);
\clip(-5.,-35.) rectangle (38.,8.);
\fill[color=sqsqsq,fill=sqsqsq,fill opacity=0.1] (0.,-20.) -- (10.,-20.) -- (10.,-30.) -- (0.,-30.) -- cycle;
\draw [->] (0.,-30.) -- (0.,5.);
\draw [->] (0.,-30.) -- (35.,-30.);
\draw (0.,0.)-- (30.,-30.);
\draw (8,-30) node[anchor=north west] {$m$};
\draw (-5,-18) node[anchor=north west] {$m$};
\draw (17.5,-30) node[anchor=north west] {$2m$};
\draw (-6.3,-8) node[anchor=north west] {$2m$};
\draw (27,-30) node[anchor=north west] {$3m$};
\draw (-6.3,2) node[anchor=north west] {$3m$};
\draw (35,-28) node[anchor=north west] {$x$};
\draw (-2,9) node[anchor=north west] {$y$};
\draw [color=sqsqsq] (0.,-20.)-- (10.,-20.);
\draw [color=sqsqsq] (10.,-20.)-- (10.,-30.);
\draw [color=sqsqsq] (10.,-30.)-- (0.,-30.);
\draw [color=sqsqsq] (0.,-30.)-- (0.,-20.);
\end{tikzpicture}

\end{center}
 \medskip
Therefore, 
$$\limsup_m \frac{m\ell_m}{v_m}= 2.$$

Consequently, from Cases~1 and 2 we obtain $\delta(S)\geq \frac{6}{5}$.
\end{proof}

 \begin{theorem}\label{theorem:d3}
Let $S$ be a del Pezzo surface of degree $3$. Then $\delta(S)\geq \frac{36}{31}$.
\end{theorem}

\begin{proof}
Let $C$ be the unique cubic on $\mathbb{P}^2$ that passes through the points $\pi(M_1),\ldots,\pi(M_{6})$ and that is singular at the point~$q$. Let $G$ be the effective divisor of a section in $\mathcal{L}_m$.

\medskip
\emph{Case 1. The curve $C$ is irreducible.}
\medskip

In this case, the Zariski tangent term of the defining polynomial of $C$ on $U$ may be assumed to contain the monomial $xy$.

 We write $G=m'C+\Omega$, where $m'$ is a non-negative integer not bigger than $m$ and $\Omega$ is an effective divisor whose support does not contain $C$. Since
$$9(m-m')=C\cdot\Omega\geq 6(m-m')+2\mult_q(\Omega),$$
we have
\[\mult_q(\Omega)\leq \frac{3}{2}(m-m').\]
This implies that  the injective map $\iota_m$ in \eqref{index-map2} can be chosen in such a way that its image  is contained in the set
\[\bigcup_{i=0}^{m}\left\{x^{i+n_1}y^{i+n_2}~|~0\leq n_1, n_2 \mbox{ and } 0\leq n_1+n_2\leq \frac{3}{2}(m-i )\right\}.\]
These monomials are plotted in the shade area as below:
\medskip

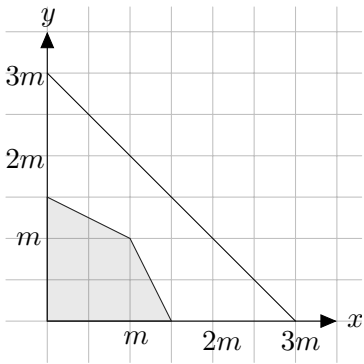
\begin{center}
\begin{tikzpicture}[scale=0.55, line cap=round,line join=round,>=triangle 45,x=0.2cm,y=0.2cm]
\draw [color=cqcqcq,, xstep=1.0cm,ystep=1.0cm] (-5.,-35.) grid (38.,8.);
\clip(-5.,-35.) rectangle (38.,8.);
\fill[color=sqsqsq,fill=sqsqsq,fill opacity=0.1] (0.,-15.) -- (10.,-20.) -- (15.,-30.) -- (0.,-30.) -- cycle;
\draw [->] (0.,-30.) -- (0.,5.);
\draw [->] (0.,-30.) -- (35.,-30.);
\draw (0.,0.)-- (30.,-30.);
\draw (8,-30) node[anchor=north west] {$m$};
\draw (-5,-18) node[anchor=north west] {$m$};
\draw (17.5,-30) node[anchor=north west] {$2m$};
\draw (-6.3,-8) node[anchor=north west] {$2m$};
\draw (27,-30) node[anchor=north west] {$3m$};
\draw (-6.3,2) node[anchor=north west] {$3m$};
\draw (35,-28) node[anchor=north west] {$x$};
\draw (-2,9) node[anchor=north west] {$y$};
\draw [color=sqsqsq] (0.,-15.)-- (10.,-20.);
\draw [color=sqsqsq] (10.,-20.)-- (15.,-30.);
\draw [color=sqsqsq] (15.,-30.)-- (0.,-30.);
\draw [color=sqsqsq] (0.,-30.)-- (0.,-15.);
\end{tikzpicture}

\end{center}
We then obtain $$\limsup_m \frac{m\ell_m}{v_m}= \frac{12}{7}$$ from Figure~\ref{table}.

\medskip
\emph{Case 2. The curve  $C$ is reducible.}
\medskip

The cubic curve $C$ then consists of either one line and one irreducible conic or three lines.

\medskip
\emph{Subcase 1. The curve  $C$ consists of a line $L$ and an irreducible conic $Q$.}
\medskip

In this subcase, we may  assume that the Zariski tangent term of the defining polynomial of~$L$ on $U$ contains the monomial $x$ and that of $Q$ contains the monomial $y$.

We write $G=m_1L+m_2Q+\Omega$, where $m_1$ and $m_2$ are non-negative integers and $\Omega$ is an effective divisor whose support contains neither $L$ nor $Q$.

We first suppose that $L$ passes through exactly one of the points $\pi(M_1),\ldots,\pi(M_{6})$. From
\[6m=Q\cdot \left( m_1L+m_2Q+\Omega\right )\geq 2m_1+4m_2+5(m-m_2)+\mult_q\Omega,\]
\[3m=L\cdot \left( m_1L+m_2Q+\Omega\right )\geq m_1+2m_2+(m-m_1)+\mult_q\Omega, \]
we obtain 
$m_2\leq m$,  $2m_1-m_2\leq m$ and 
\[\mult_q\Omega\leq \min\left\{
 m-2m_1+m_2, 
2m-2m_2\right\}.\]
Therefore, the injective map $\iota_m$ in \eqref{index-map2} can be defined in such a way that its image sits in the set
\[\bigcup_{\substack{0\leq m_2\leq m\\ 2m_1- m_2\leq m}}
\left\{x^{m_1+n_1}y^{m_2+n_2}~|~0\leq n_1, n_2 \mbox{ and } 0\leq n_1+n_2\leq \min\left\{
 \aligned
& m-2m_1+m_2\\
& 2m-2m_2\\
\endaligned  \right\} \right\}
.\]
It is easy to check that the monomials in this set are plotted in the shade area of the following diagram:
 \medskip

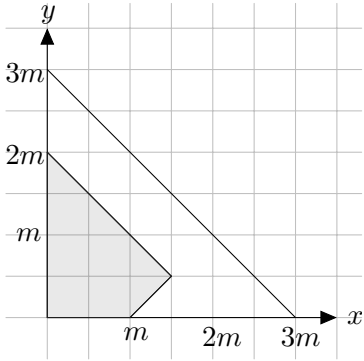
\begin{center}
\begin{tikzpicture}[scale=0.55, line cap=round,line join=round,>=triangle 45,x=0.2cm,y=0.2cm]
\draw [color=cqcqcq,, xstep=1cm,ystep=1cm] (-5.,-35.) grid (38.,8.);
\clip(-5.,-35.) rectangle (38.,8.);
\fill[color=sqsqsq,fill=sqsqsq,fill opacity=0.1] (10.,-30.) -- (15.,-25.) -- (0.,-10.) -- (0.,-30.) -- cycle;
\draw [->] (0.,-30.) -- (0.,5.);
\draw [->] (0.,-30.) -- (35.,-30.);
\draw (0.,0.)-- (30.,-30.);
\draw (10.,-30.)-- (15.,-25.);
\draw (15.,-25.)-- (10.,-20.);
\draw (10.,-20.)-- (0.,-10.);
\draw [color=sqsqsq] (10.,-30.)-- (15.,-25.);
\draw [color=sqsqsq] (15.,-25.)-- (0.,-10.);
\draw [color=sqsqsq] (0.,-10.)-- (0.,-30.);
\draw [color=sqsqsq] (0.,-30.)-- (10.,-30.);
\draw (8,-30) node[anchor=north west] {$m$};
\draw (-5,-18) node[anchor=north west] {$m$};
\draw (17.5,-30) node[anchor=north west] {$2m$};
\draw (-6.3,-8) node[anchor=north west] {$2m$};
\draw (27,-30) node[anchor=north west] {$3m$};
\draw (-6.3,2) node[anchor=north west] {$3m$};
\draw (35,-28) node[anchor=north west] {$x$};
\draw (-2,9) node[anchor=north west] {$y$};
\end{tikzpicture}

\end{center}
 \medskip
 
We now suppose that $L$ passes through two points of $\pi(M_1),\ldots,\pi(M_{6})$. Then, from the inequalities
\[6m=Q\cdot \left( m_1L+m_2Q+\Omega\right )\geq 2m_1+4m_2+4(m-m_2)+\mult_q\Omega,\]
\[3m=L\cdot \left( m_1L+m_2Q+\Omega\right )\geq m_1+2m_2+2(m-m_1)+\mult_q\Omega, \]
we obtain 
$m_1\leq m$,  $2m_2-m_1\leq m$ and 
\[\mult_q\Omega\leq \min\left\{
 m-2m_2+m_1, 
2m-2m_1\right\}.\]
As the previous case, the injective map $\iota_m$ in \eqref{index-map2}  can be chosen in a way that its image is contained in the set
\[\bigcup_{\substack{0\leq m_1\leq m\\ 2m_2- m_1\leq m}}
\left\{x^{m_1+n_1}y^{m_2+n_2}~|~0\leq n_1, n_2 \mbox{ and } 0\leq n_1+n_2\leq \min\left\{
 \aligned
& m-2m_2+m_1\\
& 2m-2m_1\\
\endaligned  \right\} \right\}
.\]
The monomials in this set are plotted in the shade area of the following diagram:
 \medskip
 
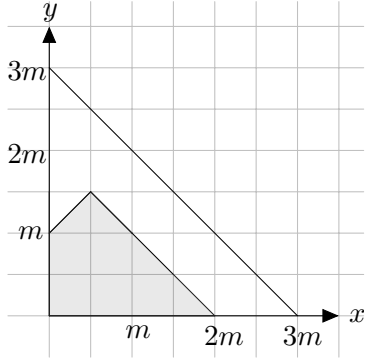
\begin{center}
\begin{tikzpicture}[scale=0.55, line cap=round,line join=round,>=triangle 45,x=0.2cm,y=0.2cm]
\draw [color=cqcqcq,, xstep=1cm,ystep=1cm] (-5.,-35.) grid (38.,8.);
\clip(-5.,-35.) rectangle (38.,8.);
\fill[color=sqsqsq,fill=sqsqsq,fill opacity=0.1] (0.,-30.) -- (0.,-20) -- (5.,-15.) -- (20.,-30.) -- cycle;
\draw [->] (0.,-30.) -- (0.,5.);
\draw [->] (0.,-30.) -- (35.,-30.);
\draw (0.,0.)-- (30.,-30.);
\draw (15.,-25.)-- (10.,-20.);
\draw (10.,-20.)-- (5.,-15.);
\draw (0.,-20.)-- (5.,-15.);
\draw [color=sqsqsq] (20.,-30.)-- (15.,-25.);
\draw [color=sqsqsq] (0.,-20.)-- (5.,-15.);
\draw [color=sqsqsq] (0.,-30.)-- (0.,-15.);
\draw [color=sqsqsq] (0.,-30.)-- (20.,-30.);
\draw  [color=sqsqsq] (10.,-20.)-- (5.,-15.);
\draw (8,-30) node[anchor=north west] {$m$};
\draw (-5,-18) node[anchor=north west] {$m$};
\draw (17.5,-30) node[anchor=north west] {$2m$};
\draw (-6.3,-8) node[anchor=north west] {$2m$};
\draw (27,-30) node[anchor=north west] {$3m$};
\draw (-6.3,2) node[anchor=north west] {$3m$};
\draw (35,-28) node[anchor=north west] {$x$};
\draw (-2,9) node[anchor=north west] {$y$};
\end{tikzpicture}

\end{center}
 \medskip
Therefore, Figure~\ref{table} implies 
$$\limsup_m \frac{m\ell_m}{v_m}= \frac{36}{31}.$$

\medskip
\emph{Subcase 2.  The curve $C$ consists of three lines $L_1$, $L_2$, $L_3$ with $\mult_q(C)=2$.}
\medskip

We may assume that $q$ is the intersection point of $L_1$ and $L_2$. In addition, we may assume that $L_1$ is defined by $x=0$  and $L_2$ by $y=0$.

We write $G=m_1L_1+m_2L_2+m_3L_3+\Omega$, where $m_1, m_2, m_3$ are non-negative integers and $\Omega$ is an effective divisor whose support contains none of $L_1, L_2, L_3$. From the inequalities 
\[3m=L_1\cdot (m_1L_1+m_2L_2+m_3L_3+\Omega)\geq m_1+m_2+m_3+2(m-m_1)+\mult_q(\Omega),\] 
\[3m=L_2\cdot (m_1L_1+m_2L_2+m_3L_3+\Omega)\geq m_1+m_2+m_3+2(m-m_2)+\mult_q(\Omega),\]
\[3m=L_3\cdot (m_1L_1+m_2L_2+m_3L_3+\Omega)\geq m_1+m_2+m_3+2(m-m_3),\]  
we see that 
the injective map $\iota_m$ in \eqref{index-map2}  can be defined in such a way that its image is contained in the set
\[\bigcup_{\substack{0\leq m_1,m_2,m_3\leq m\\ m_1+m_2\leq m_3+m}}
\left\{x^{m_1+n_1}y^{m_2+n_2}~|~0\leq n_1, n_2 \mbox{ and }  0\leq n_1+n_2\leq \min\left\{
\aligned
&m-m_1+m_2-m_3\\
& m+m_1-m_2-m_3\\
\endaligned  \right\} \right\}
.\]
The monomials in this set are plotted in the shade area of the following diagram:
 \medskip
 
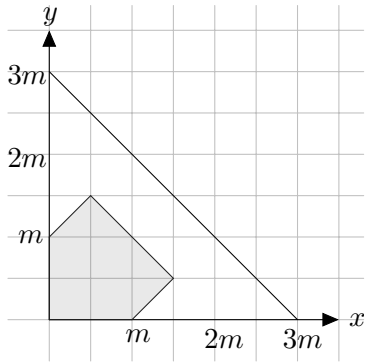
\begin{center}
\begin{tikzpicture}[scale=0.55, line cap=round,line join=round,>=triangle 45,x=0.2cm,y=0.2cm]
\draw [color=cqcqcq,, xstep=1.0cm,ystep=1.0cm] (-5.,-35.) grid (38.,8.);
\clip(-5.,-35.) rectangle (38.,8.);
\fill[color=sqsqsq,fill=sqsqsq,fill opacity=0.1] (0.,-20.) -- (5.,-15.) -- (15.,-25.) -- (10.,-30.) -- (0.,-30.) -- cycle;
\draw [->] (0.,-30.) -- (0.,5.);
\draw [->] (0.,-30.) -- (35.,-30.);
\draw (0.,0.)-- (30.,-30.);
\draw (8,-30) node[anchor=north west] {$m$};
\draw (-5,-18) node[anchor=north west] {$m$};
\draw (17.5,-30) node[anchor=north west] {$2m$};
\draw (-6.3,-8) node[anchor=north west] {$2m$};
\draw (27,-30) node[anchor=north west] {$3m$};
\draw (-6.3,2) node[anchor=north west] {$3m$};
\draw (35,-28) node[anchor=north west] {$x$};
\draw (-2,9) node[anchor=north west] {$y$};
\draw [color=sqsqsq] (0.,-20.)-- (5.,-15.);
\draw [color=sqsqsq] (5.,-15.)-- (15.,-25.);
\draw [color=sqsqsq] (15.,-25.)-- (10.,-30.);
\draw [color=sqsqsq] (10.,-30.)-- (0.,-30.);
\draw [color=sqsqsq] (0.,-30.)-- (0.,-20.);
\end{tikzpicture}

\end{center}
 \medskip
Therefore,
$$\limsup_m \frac{m\ell_m}{v_m}= \frac{18}{11}.$$

\medskip
\emph{Subcase 3. The curve $C$ consists of three lines $L_1$, $L_2$, $L_3$ with $\mult_q(C)=3$.}
\medskip

The line $L_1$ can be assumed to be defined by $x=0$, $L_2$ by $y=0$ and $L_3$ by $x+y=0$.

We write $G=m_1L_1+m_2L_2+m_3L_3+\Omega$, where $m_1, m_2, m_3$ are non-negative integers and $\Omega$ is an effective divisor whose support contains none of $L_1, L_2, L_3$. 
The three inequalities 

\[3m=L_1\cdot (m_1L_1+m_2L_2+m_3L_3+\Omega)\geq m_1+m_2+m_3+2(m-m_1)+\mult_q(\Omega),\] 
\[3m=L_2\cdot (m_1L_1+m_2L_2+m_3L_3+\Omega)\geq m_1+m_2+m_3+2(m-m_2)+\mult_q(\Omega),\]
\[3m=L_3\cdot (m_1L_1+m_2L_2+m_3L_3+\Omega)\geq m_1+m_2+m_3+2(m-m_3)+\mult_q(\Omega)\]  
show that the injective map $\iota_m$ in \eqref{index-map2} can be arranged in such a way that
 its  image is  contained in the set
\[\bigcup_{0\leq m_1, m_2, m_3\leq m}
\left\{x^{m_1+m_3+n_1}y^{m_2+n_2}~|~0\leq n_1, n_2 \mbox{ and }  0\leq n_1+n_2\leq  \min\left\{
\aligned
&m+m_1-m_2-m_3\\
& m-m_1+m_2-m_3\\
& m-m_1-m_2+m_3\\
\endaligned  \right\}  \right\}
.\]
The monomials in this set sit in the shade area in the diagram below:
 \medskip
 
\begin{center}
\begin{tikzpicture}[scale=0.55, line cap=round,line join=round,>=triangle 45,x=0.2cm,y=0.2cm]
\draw [color=cqcqcq,, xstep=1.0cm,ystep=1.0cm] (-5.,-35.) grid (40.,8.);
\clip(-5.,-35.) rectangle (40.,8.);
\fill[color=sqsqsq,fill=sqsqsq,fill opacity=0.1] (10.,-30.) -- (20.,-20.) -- (0.,-20.) -- (0.,-30.) -- cycle;
\draw [->] (0.,-30.) -- (0.,5.);
\draw [->] (0.,-30.) -- (35.,-30.);
\draw (0.,0.)-- (30.,-30.);
\draw (8,-30) node[anchor=north west] {$m$};
\draw (-5,-18) node[anchor=north west] {$m$};
\draw (17.5,-30) node[anchor=north west] {$2m$};
\draw (-6.3,-8) node[anchor=north west] {$2m$};
\draw (27,-30) node[anchor=north west] {$3m$};
\draw (-6.3,2) node[anchor=north west] {$3m$};
\draw (35,-28) node[anchor=north west] {$x$};
\draw (-2,9) node[anchor=north west] {$y$};
\draw [color=sqsqsq] (10.,-30.)-- (20.,-20.);
\draw [color=sqsqsq] (20.,-20.)-- (0.,-20.);
\draw [color=sqsqsq] (0.,-20.)-- (0.,-30.);
\draw [color=sqsqsq] (0.,-30.)-- (10.,-30.);
\end{tikzpicture}

\end{center}
 \medskip
Figure~\ref{table} then implies 
$$\limsup_m \frac{m\ell_m}{v_m}= \frac{9}{7}.$$

Consequently, Cases~1 and 2 imply $\delta(S)\geq \frac{36}{31}$.
\end{proof}

 \begin{theorem}\label{theorem:d4}
Let $S$ be a del Pezzo surface of degree $4$. Then $\delta(S)\geq \frac{12}{11}$.
\end{theorem}

\begin{proof} 
Note that at most two $(-1)$-curves can pass through a given point $p$ on $S$. 

Let $G$ be the effective divisor of a section in $\mathcal{L}_m$. 

\medskip
\emph{Case 1.  There is no $(-1)$-curve that passes through the point $p$.}
\medskip

In this case,  there is  an irreducible cubic   curve  $C$  on $\mathbb{P}^2$ that passes through the points $\pi(M_1),\ldots,\pi(M_{5})$ and that is singular at the point~$q$. We may assume that the Zariski tangent term of the defining polynomial of $C$ on $U$ contains the monomial $xy$.

  We write $G=m'C+\Omega$, where $m'$ is a non-negative integer not bigger than $m$ and $\Omega$ is an effective divisor whose support does not contain $C$. From the inequality 
$$9(m-m')=C\cdot\Omega\geq 5(m-m')+2\mult_q(\Omega),$$ we obtain
\[\mult_q(\Omega)\leq 2(m-m').\]
Therefore,  the injective map $\iota_m$ in \eqref{index-map2} can be chosen in such a way that its image  is contained in the set
\[\bigcup_{i=0}^m
\left\{x^{i+n_1}y^{i+n_2}~|~0\leq n_1, n_2 \mbox{ and } 0\leq n_1+n_2\leq 2(m-i) \right\}
.\]
The monomials in this set correspond to the integral points in the shade area below:
\medskip

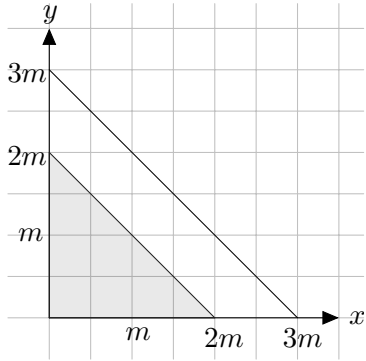
\begin{center}
\begin{tikzpicture}[scale=0.55, line cap=round,line join=round,>=triangle 45,x=0.2cm,y=0.2cm]
\draw [color=cqcqcq,, xstep=1.0cm,ystep=1.0cm] (-5.,-35.) grid (38.,8.);
\clip(-5.,-35.) rectangle (38.,8.);
\fill[color=sqsqsq,fill=sqsqsq,fill opacity=0.1] (20.,-30.) -- (0.,-10.) -- (0.,-30.) -- cycle;
\draw [->] (0.,-30.) -- (0.,5.);
\draw [->] (0.,-30.) -- (35.,-30.);
\draw (0.,0.)-- (30.,-30.);
\draw (8,-30) node[anchor=north west] {$m$};
\draw (-5,-18) node[anchor=north west] {$m$};
\draw (17.5,-30) node[anchor=north west] {$2m$};
\draw (-6.3,-8) node[anchor=north west] {$2m$};
\draw (27,-30) node[anchor=north west] {$3m$};
\draw (-6.3,2) node[anchor=north west] {$3m$};
\draw (35,-28) node[anchor=north west] {$x$};
\draw (-2,9) node[anchor=north west] {$y$};
\draw [color=sqsqsq] (20.,-30.)-- (0.,-10.);
\draw [color=sqsqsq] (0.,-10.)-- (0.,-30.);
\draw [color=sqsqsq] (0.,-30.)-- (20.,-30.);
\end{tikzpicture}
\end{center}
 \medskip
Therefore, Figure~\ref{table} implies 
$$\limsup_m \frac{m\ell_m}{v_m}=\frac{3}{2}.$$

\medskip

\emph{Case 2. There is only one $(-1)$-curve that passes through the point $p$.}
\medskip

In this case, there is either a line  passing through the point $q$ and two of the points~$\pi(M_1),\ldots,\pi(M_{5})$
or an irreducible conic passing through  $q$ and $\pi(M_1),\ldots,\pi(M_{5})$. We consider the former case only since the latter can be dealt in almost the same way. Let $L$ be the line in the former case.
We may assume that it passes through $\pi(M_1)$ and $\pi(M_2)$. Also, there is an irreducible conic $Q$ that passes through the points $q$ and 
$\pi(M_2),\ldots,\pi(M_{5})$.
We write $G=m_1L+m_2Q+\Omega$, where $m_1, m_2$ are non-negative integers  and $\Omega$ is an effective divisor whose support contains neither $L$ nor $Q$.

From the inequalities
\[3m=L\cdot(m_1L+m_2Q+\Omega)\geq m_1+2m_2+2(m-m_1)+\mult_q(\Omega),\]
\[6m=Q\cdot(m_1L+m_2Q+\Omega)\geq 2m_1+4m_2+4(m-m_2)+\mult_q(\Omega),\]
we obtain  $$ \mult_q(\Omega)\leq \min\{m+m_1-2m_2, 2m-2m_1\}.$$
This implies that
\[\mult_q(G)= m_1+m_2+\mult_q(\Omega)\leq \min\{m+2m_1-m_2, 2m-m_1+m_2\}\leq 2m.\]
Therefore, the image of the injective map $\iota_m$ in \eqref{index-map2} is always contained in the same set as in the previous case, and hence 
$$\limsup_m \frac{m\ell_m}{v_m}=\frac{3}{2}.$$

\medskip
\emph{Case 3.  There are two $(-1)$-curves  that pass through the point $p$.}
\medskip

There are two distinct lines $L_1$,  $L_2$ passing through the point $q$ and two of the points $\pi(M_1),\ldots,\pi(M_{5})$.  We may assume that $L_1$ passes through $\pi(M_1)$, $\pi(M_2)$ and that $L_2$ passes through $\pi(M_3)$, $\pi(M_4)$. Let $L_3$ be the line determined by $q$ and $\pi(M_5)$.

By suitable coordinate changes, we may assume that the line $L_1$ is defined by $x=0$,  $L_2$ by~$y=0$ and $L_3$ by $x+y=0$. We then see that 
$$t_1=1, \ t_2= x, \ t_3=y, \ t_4=xy, \ t_5=xy(x+y)$$
form a basis  for $\mathcal{T}_1$. Furthermore, we may take 
$$\mathbf{x}_{1,1}=1, \ \mathbf{x}_{1,2}=x, \ \mathbf{x}_{1,3}=y,\ \mathbf{x}_{1,4}=xy,\ \mathbf{x}_{1,5}=x^2y.$$ 
 Then the set \[\mathcal{C}_m=\left\{ \prod_{i=1}^{5} \mathbf{x}_{1,i}^{n_i} \ \Big| \ \mbox{$n_i$ are non-negative integers with $n_1+\cdots+n_{5}=m$}\right\}\]
consists of the monomials corresponding the integral  points in the shade area of the following diagram:
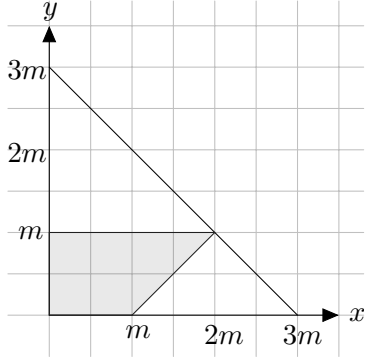
\begin{center}
 \medskip

\begin{tikzpicture}[scale=0.55, line cap=round,line join=round,>=triangle 45,x=0.2cm,y=0.2cm]
\draw [color=cqcqcq,, xstep=1.0cm,ystep=1.0cm] (-5.,-35.) grid (38.,8.);
\clip(-5.,-35.) rectangle (38.,8.);
\fill[color=sqsqsq,fill=sqsqsq,fill opacity=0.1] (10.,-30.) -- (20.,-20.) -- (0.,-20.) -- (0.,-30.) -- cycle;
\draw [->] (0.,-30.) -- (0.,5.);
\draw [->] (0.,-30.) -- (35.,-30.);
\draw (0.,0.)-- (30.,-30.);
\draw (8,-30) node[anchor=north west] {$m$};
\draw (-5,-18) node[anchor=north west] {$m$};
\draw (17.5,-30) node[anchor=north west] {$2m$};
\draw (-6.3,-8) node[anchor=north west] {$2m$};
\draw (27,-30) node[anchor=north west] {$3m$};
\draw (-6.3,2) node[anchor=north west] {$3m$};
\draw (35,-28) node[anchor=north west] {$x$};
\draw (-2,9) node[anchor=north west] {$y$};
\draw [color=sqsqsq] (10.,-30.)-- (20.,-20.);
\draw [color=sqsqsq] (20.,-20.)-- (0.,-20.);
\draw [color=sqsqsq] (0.,-20.)-- (0.,-30.);
\draw [color=sqsqsq] (0.,-30.)-- (10.,-30.);
\end{tikzpicture}

\end{center}
 \medskip

We now write $G=m_1L_1+m_2L_2+m_3L_3+\Omega$, where $m_1, m_2, m_3$ are non-negative integers  and $\Omega$ is an effective divisor whose support contains none of $L_1$, $L_2$, $L_3$. The inequities 
\[3m=L_1\cdot(m_1L_1+m_2L_2+m_3L_3+\Omega)\geq m_1+m_2+m_3+2(m-m_1)+\mult_q(\Omega),\]
\[3m=L_2\cdot(m_1L_1+m_2L_2+m_3L_3+\Omega)\geq m_1+m_2+m_3+2(m-m_2)+\mult_q(\Omega),\]
\[3m=L_3\cdot(m_1L_1+m_2L_2+m_3L_3+\Omega)\geq m_1+m_2+m_3+(m-m_3)+\mult_q(\Omega)\]
imply
\[\mult_q(\Omega)\leq \min\{m+m_1-m_2-m_3, m-m_1+m_2-m_3, 2m-m_1-m_2\}.\]
This shows that the images $\iota_m(u_1),\ldots, \iota_m(u_a)$ can be plotted only above or on the line joining $(m,0)$ and $(2m, m)$.
Since $\deg (G)=3m$, they must sit only below or on the line joining $(3m,0)$ and $(0, 3m)$.
Consequently, the maximum value $v_m$ can be attained when the image of  the injective map $\iota_m$ in \eqref{index-map2}
are contained in the set of the monomials corresponding the integral  points in the shade area of the following diagram:
\medskip

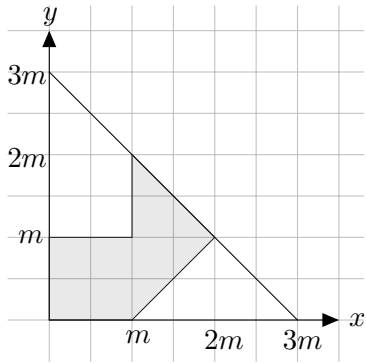
\begin{center}
\begin{tikzpicture}[scale=0.55, line cap=round,line join=round,>=triangle 45,x=0.2cm,y=0.2cm]
\draw [color=cqcqcq,, xstep=1.0cm,ystep=1.0cm] (-5.,-35.) grid (38.,8.);
\clip(-5.,-35.) rectangle (38.,8.);
\fill[color=sqsqsq,fill=sqsqsq,fill opacity=0.1] (0.,-30.) -- (10.,-30.) -- (20.,-20.) -- (10.,-10.) -- (10.,-20.) -- (0.,-20.) -- cycle;
\draw [->] (0.,-30.) -- (0.,5.);
\draw [->] (0.,-30.) -- (35.,-30.);
\draw (0.,0.)-- (30.,-30.);
\draw (8,-30) node[anchor=north west] {$m$};
\draw (-5,-18) node[anchor=north west] {$m$};
\draw (17.5,-30) node[anchor=north west] {$2m$};
\draw (-6.3,-8) node[anchor=north west] {$2m$};
\draw (27,-30) node[anchor=north west] {$3m$};
\draw (-6.3,2) node[anchor=north west] {$3m$};
\draw (35,-28) node[anchor=north west] {$x$};
\draw (-2,9) node[anchor=north west] {$y$};
\draw [color=sqsqsq] (0.,-30.)-- (10.,-30.);
\draw [color=sqsqsq] (10.,-30.)-- (20.,-20.);
\draw [color=sqsqsq] (20.,-20.)-- (10.,-10.);
\draw [color=sqsqsq] (10.,-10.)-- (10.,-20.);
\draw [color=sqsqsq] (10.,-20.)-- (0.,-20.);
\draw [color=sqsqsq] (0.,-20.)-- (0.,-30.);
\end{tikzpicture}
\end{center}
 \medskip
Therefore,
$$\limsup_m \frac{m\ell_m}{v_m}= \frac{12}{11}.$$

Consequently, Cases~1, 2 and 3 imply $\delta(S)\geq \frac{12}{11}$.

\end{proof}

 \begin{theorem}\label{theorem:d5}
Let $S$ be a del Pezzo surface of degree $5$. Then $\delta(S)\geq \frac{15}{14}$.
\end{theorem}

\begin{proof}
For a point $p$ on $S$, we may or may not have a member in $|-K_S|$ that has multiplicity~$3$ at $p$.

\medskip
\emph{Case 1.  No divisor in $|-K_S|$ has multiplicity $3$ at $p$.}
\medskip

In this case, the monomials  $$t_1=1, \ t_2=x, \ t_3=y, \ t_4=x^2, \ t_5=xy, \ t_6=y^2$$ form a basis for $\mathcal{T}_1$. Therefore, 
$$\mathbf{x}_{1,1}=1, \ \mathbf{x}_{1,2}=x, \ \mathbf{x}_{1,3}=y,\ \mathbf{x}_{1,4}=x^2,\ \mathbf{x}_{1,5}=xy,\ \mathbf{x}_{1,6}=y^2.$$  Then we obtain the  set 
\[\mathcal{C}_m=
\left\{x^{n_1}y^{n_2}~|~0\leq n_1, n_2 \mbox{ and } 0\leq n_1+n_2\leq 2m \right\}
.\]
Let $G$ be the effective divisor of a section in $\mathcal{L}_m$ and let $H$ be a general curve  of degree $3$ on $\mathbb{P}^2$ that has multiplicity $2$ at $q$ and passes through $\pi(M_1)$, $\pi(M_2)$, $\pi(M_3)$ and $\pi(M_4)$. Then the inequality 
\[9m\geq \mult_q(H)\mult_q(G)+\sum_{i=1}^4\mult_{\pi(M_i)}(H)\mult_{\pi(M_i)}(G)\geq 4m+2\mult_q(G)\]
shows that the maximum value $v_m$ can be attained when the image of  the injective map $\iota_m$ in \eqref{index-map2}
is contained in the set of the monomials corresponding to the integral  points in the shade area of the following diagram:
 \medskip

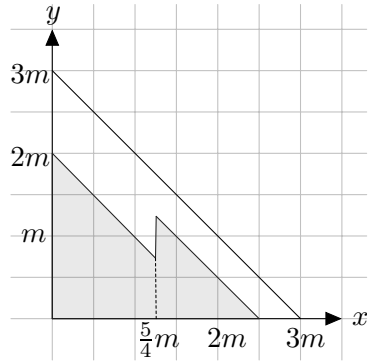
\begin{center}
\begin{tikzpicture}[scale=0.55, line cap=round,line join=round,>=triangle 45,x=0.2cm,y=0.2cm]
\draw [color=cqcqcq,, xstep=1.0cm,ystep=1.0cm] (-5.,-35.) grid (38.,8.);
\clip(-5.,-35.) rectangle (38.,8.);
\fill[color=sqsqsq,fill=sqsqsq,fill opacity=0.1] (0.,-10.) -- (12.475239083380389,-22.67985609160773) -- (12.559480994579683,-17.625341419650184) -- (25.,-30.) -- (0.,-30.) -- cycle;
\draw [->] (0.,-30.) -- (0.,5.);
\draw [->] (0.,-30.) -- (35.,-30.);
\draw (0.,0.)-- (30.,-30.);
\draw (-5,-18) node[anchor=north west] {$m$};
\draw (17.5,-30) node[anchor=north west] {$2m$};
\draw (-6.3,-8) node[anchor=north west] {$2m$};
\draw (27,-30) node[anchor=north west] {$3m$};
\draw (-6.3,2) node[anchor=north west] {$3m$};
\draw (35,-28) node[anchor=north west] {$x$};
\draw (-2,9) node[anchor=north west] {$y$};
\draw [color=sqsqsq] (0.,-10.)-- (12.475239083380389,-22.67985609160773);
\draw [color=sqsqsq] (12.475239083380389,-22.67985609160773)-- (12.559480994579683,-17.625341419650184);
\draw [color=sqsqsq] (12.559480994579683,-17.625341419650184)-- (25.,-30.);
\draw [color=sqsqsq] (25.,-30.)-- (0.,-30.);
\draw [color=sqsqsq] (0.,-30.)-- (0.,-10.);
\draw [dash pattern=on 1pt off 1pt] (12.475239083380387,-22.67985609160773)-- (12.559480994579683,-30.);
\draw (9,-29.2) node[anchor=north west] {$\frac{5}{4}m$};
\end{tikzpicture}
\end{center}
 \medskip
To be precise, the monomials of the set $\mathcal{C}_m$ correspond to the integral points below and on the line  joining $(2m,0)$ and $(0,2m)$. The images $\iota_m(u_1),\ldots, \iota_m(u_a)$ sit in the shade area above the line  joining $(2m,0)$ and $(0,2m)$.

Then, Figure~\ref{table} shows that
$$\limsup_m \frac{m\ell_m}{v_m}=\frac{15}{14}.$$

\medskip
\emph{Case 2.  There is a divisor $C$   in $|-K_S|$ that has multiplicity~$3$ at $p$.}
\medskip

The curve $\pi(C)$ consists of three lines $L_1$, $L_2$, $L_3$ passing through the point $q$. We may assume that the points $\pi(M_1)$ and $\pi(M_2)$ belong to $L_1$. Then the points $\pi(M_3)$ and $\pi(M_4)$ belong to the union of $L_2$ and $L_3$.
Furthermore, we may assume that $L_1$ is defined by $x=0$, $L_2$ by $y=0$ and $L_3$ by $x+y=0$.  Then   
$$t_1=1, \ t_2=x, \ t_3=y, \ t_4=x^2, \ t_5=xy, \ t_6=xy(x+y)$$ form a basis for $\mathcal{T}_1$. We may take
$$\mathbf{x}_{1,1}=1, \ \mathbf{x}_{1,2}=x, \ \mathbf{x}_{1,3}=y,\ \mathbf{x}_{1,4}=x^2,\ \mathbf{x}_{1,5}=xy,\ \mathbf{x}_{1,6}=xy^2.$$
The set $\mathcal{C}_m$
consists of the monomials corresponding the integral  points in the shade area of the following diagram:
 \medskip
 
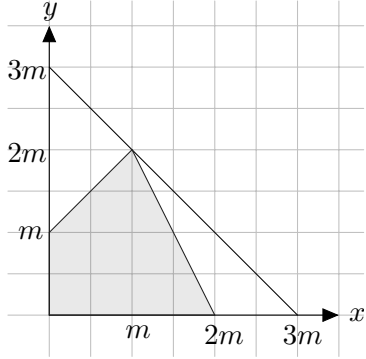
\begin{center}
\begin{tikzpicture}[scale=0.55, line cap=round,line join=round,>=triangle 45,x=0.2cm,y=0.2cm]
\draw [color=cqcqcq,, xstep=1.0cm,ystep=1.0cm] (-5.,-35.) grid (38.,8.);
\clip(-5.,-35.) rectangle (38.,8.);
\fill[color=sqsqsq,fill=sqsqsq,fill opacity=0.1] (20.,-30.) -- (10.,-10.) -- (0.,-20.) -- (0.,-30.) -- cycle;
\draw [->] (0.,-30.) -- (0.,5.);
\draw [->] (0.,-30.) -- (35.,-30.);
\draw (0.,0.)-- (30.,-30.);
\draw (8,-30) node[anchor=north west] {$m$};
\draw (-5,-18) node[anchor=north west] {$m$};
\draw (17.5,-30) node[anchor=north west] {$2m$};
\draw (-6.3,-8) node[anchor=north west] {$2m$};
\draw (27,-30) node[anchor=north west] {$3m$};
\draw (-6.3,2) node[anchor=north west] {$3m$};
\draw (35,-28) node[anchor=north west] {$x$};
\draw (-2,9) node[anchor=north west] {$y$};
\draw [color=sqsqsq] (20.,-30.)-- (10.,-10.);
\draw [color=sqsqsq] (10.,-10.)-- (0.,-20.);
\draw [color=sqsqsq] (0.,-20.)-- (0.,-30.);
\draw [color=sqsqsq] (0.,-30.)-- (20.,-30.);
\end{tikzpicture}
\end{center}
 \medskip
 
Note that the number of monomials in this set is exactly $\ell_m$. Therefore, $a=0$ in Step~2 and Figure~\ref{table} shows that
$$\limsup_m \frac{m\ell_m}{v_m}=\frac{15}{13}.$$

Cases~1 and 2 complete the proof.
\end{proof}

\section{K-semistable del Pezzo surfaces II}

In this section, together with Section~\ref{section:semi-I}, we complete the proof of the second statement of Main Theorem.

 Let $S$ be the del Pezzo surface of degree $6$. This surface can be obtained by blowing up $\mathbb{P}^2$  at $p_1=[0:0:1]$, $p_2=[0:1:0]$ and $p_3=[1:0:0]$. Let $\phi_3:S\to \mathbb{P}^2$ be the blow-up with the exceptional curves  $E$, $F$, $G$. For a fixed positive integer $m$, the space $\mathrm{H}^0(S,\mathcal{O}_{S}(-mK_{S}))$ can be regarded as the subspace $\mathcal{L}_m$ of $\mathrm{H}^0(\mathbb{P}^2,\mathcal{O}_{\mathbb{P}^2}(3m))$ consisting of the sections vanishing at $p_1, p_2, p_3$ with order at least $m$. 
The set
\[\mathcal{B}_3=\left\{ x^ay^bz^c\ | \ a, b, c \mbox{ are  integers with } a+b+c=3m \mbox{ and } 0\leq a, b, c\leq 2m\right\}\]
forms a basis for $\mathcal{L}_m$. Let $D_{a,b,c}$ be the divisor defined by $x^ay^bz^c$ on $\mathbb{P}^2$. 
Then the divisor
\[D=\frac{1}{m(3m^2+3m+1)}\sum_{x^ay^bz^c\in\mathcal{B}_3} (\phi_3^*(D_{a,b,c})-mE-mF-mG)\]
is an anticanonical $\mathbb{Q}$-divisor of $m$-basis type on $S$. We immediately see that
\[D=L_x+L_y+L_y+E+F+G,\]
where $L_x$, $L_y$, $L_z$ are the proper transforms of the lines  on $\mathbb{P}^2$ defined by $x=0$, $y=0$, $z=0$, respectively. Therefore,
$\delta(S)\leq 1$.

\begin{theorem}\label{theorem:degree 6}
The $\delta$-invariant of the del Pezzo surface of degree $6$ is $1$.
\end{theorem}
\begin{proof}
Fix a positive integer $m$ and set $\ell_m=h^0(S, \mathcal{O}_S(-mK_S))$.
Let $\{s_1,\ldots, s_{\ell_m}\}$ be a basis of $\mathrm{H}^0(S, \mathcal{O}_S(-mK_S))$. We denote the effective divisor of the section $s_i$ by $D_i$. Put $D=\sum D_i$.
Let~$p$ be an arbitrary point on $S$.

\medskip
\emph{Case 1. The point $p$ is not an intersection point of two $(-1)$-curves.}
\medskip

In this case,  there is a birational morphism $\pi:S\to \mathbb{P}^2$ that is an isomorphism around the point $p$. We can use the exactly same method as in the previous section. In order to apply the same method, we set $q=\pi(p)$ and use the same notations as before. By a suitable coordinate change, we assume that~$q=[0:0:1]$.

\medskip
\emph{Subcase 1. The point $p$ does not lie on any $(-1$)-curve.}
\medskip

Note that $\dim_k\mathcal{T}_1=7.$
Since the three lines determined by the points $q$ and $\pi(M_i)$, $i=1,2,3$, are distinct,  by a suitable coordinate change we may  assume that $$t_1=1, \  t_2=x, \ t_3=y,\ t_4=x^2,\ t_5=xy,\ t_6=y^2,\ t_7=xy(x+y)$$
form a basis for $\mathcal{T}_1$. We may therefore take
$$\mathbf{x}_{1,1}=1, \ \mathbf{x}_{1,2}=x, \ \mathbf{x}_{1,3}=y,\ \mathbf{x}_{1,4}=x^2,\ \mathbf{x}_{1,5}=xy,\ \mathbf{x}_{1,6}=y^2,\ \mathbf{x}_{1,7}=x^2y.$$ 
 Then the set $\mathcal{C}_m$
consists of the monomials corresponding to the integral  points in the shade area of the following diagram:
 \medskip
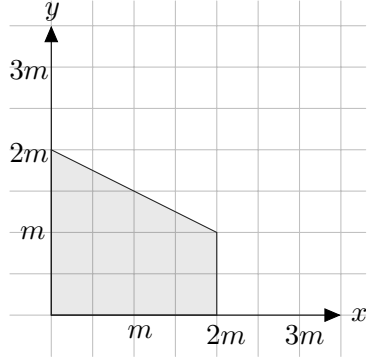
\begin{center}
\begin{tikzpicture}[scale=0.55, line cap=round,line join=round,>=triangle 45,x=0.2cm,y=0.2cm]
\draw [color=cqcqcq,, xstep=1.0cm,ystep=1.0cm] (-5.,-35.) grid (38.,8.);
\clip(-5.,-35.) rectangle (38.,8.);
\fill[color=sqsqsq,fill=sqsqsq,fill opacity=0.1] (20.,-30.) -- (20.,-20.) -- (0.,-10.) -- (0.,-30.) -- cycle;
\draw [->] (0.,-30.) -- (0.,5.);
\draw [->] (0.,-30.) -- (35.,-30.);
\draw (8,-30) node[anchor=north west] {$m$};
\draw (-5,-18) node[anchor=north west] {$m$};
\draw (17.5,-30) node[anchor=north west] {$2m$};
\draw (-6.3,-8) node[anchor=north west] {$2m$};
\draw (27,-30) node[anchor=north west] {$3m$};
\draw (-6.3,2) node[anchor=north west] {$3m$};
\draw (35,-28) node[anchor=north west] {$x$};
\draw (-2,9) node[anchor=north west] {$y$};
\draw [color=sqsqsq] (20.,-30.)-- (20.,-20.);
\draw [color=sqsqsq] (20.,-20.)-- (0.,-10.);
\draw [color=sqsqsq] (0.,-10.)-- (0.,-30.);
\draw [color=sqsqsq] (0.,-30.)-- (20.,-30.);
\end{tikzpicture}
\end{center}
 \medskip
Note that the number of monomials in this set is exactly $\ell_m$. Therefore, $a=0$ in Step~2 and  Figure~\ref{table} shows that
$$\limsup_m \frac{m\ell_m}{v_m}=\frac{9}{8}.$$

\medskip
\emph{Subcase 2. The point $p$  lies on a single $(-1$)-curve.}
\medskip

In this case, by a suitable coordinate change we may  assume that $$t_1=1, \  t_2=x, \ t_3=y,\ t_4=x^2,\ t_5=xy,\ t_6=x^2y,\ t_7=xy^2$$
form a basis for $\mathcal{T}_1$.  Then $$\mathbf{x}_{1,1}=1, \ \mathbf{x}_{1,2}=x, \ \mathbf{x}_{1,3}=y,\ \mathbf{x}_{1,4}=x^2,\ \mathbf{x}_{1,5}=xy,\ \mathbf{x}_{1,6}=x^2y,\ \mathbf{x}_{1,7}=xy^2.$$  The set $\mathcal{C}_m $
consists of the monomials corresponding to the integral  points in the shade area of the following diagram:

\medskip
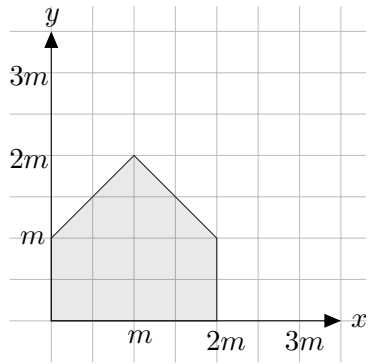
\begin{center}
\begin{tikzpicture}[scale=0.55, line cap=round,line join=round,>=triangle 45,x=0.2cm,y=0.2cm]
\draw [color=cqcqcq,, xstep=1.0cm,ystep=1.0cm] (-5.,-35.) grid (38.,8.);
\clip(-5.,-35.) rectangle (38.,8.);
\fill[color=sqsqsq,fill=sqsqsq,fill opacity=0.1] (0.,-20.) -- (10.,-10.) -- (20.,-20.) -- (20.,-30.) -- (0.,-30.) -- cycle;
\draw [->] (0.,-30.) -- (0.,5.);
\draw [->] (0.,-30.) -- (35.,-30.);
\draw (8,-30) node[anchor=north west] {$m$};
\draw (-5,-18) node[anchor=north west] {$m$};
\draw (17.5,-30) node[anchor=north west] {$2m$};
\draw (-6.3,-8) node[anchor=north west] {$2m$};
\draw (27,-30) node[anchor=north west] {$3m$};
\draw (-6.3,2) node[anchor=north west] {$3m$};
\draw (35,-28) node[anchor=north west] {$x$};
\draw (-2,9) node[anchor=north west] {$y$};\draw [color=sqsqsq] (0.,-20.)-- (10.,-10.);
\draw [color=sqsqsq] (10.,-10.)-- (20.,-20.);
\draw [color=sqsqsq] (20.,-20.)-- (20.,-30.);
\draw [color=sqsqsq] (20.,-30.)-- (0.,-30.);
\draw [color=sqsqsq] (0.,-30.)-- (0.,-20.);
\end{tikzpicture}
\end{center}
\medskip
The number of monomials in this set is exactly $\ell_m$. Therefore, $a=0$ in Step~2 and  Figure~\ref{table} shows that
$$\limsup_m \frac{m\ell_m}{v_m}= 1.$$

\medskip
\emph{Case 2. The point $p$ is  an intersection point of two $(-1)$-curves.}
\medskip

There is a birational morphism $\phi:S\to\mathbb{P}^1\times\mathbb{P}^1$  that is an isomorphism around the point $p$. The morphism $\phi$ is obtained by contracting two suitable disjoint $(-1)$-curves $M_1, M_2$ on $S$.

For an effective divisor $C\in |-mK_{S}|$, the divisor $\phi(C)$ is a curve of bidegree $(2m, 2m)$ on $\mathbb{P}^1\times \mathbb{P}^1$ which passes through the points   $\phi(M_1)$ and $\phi(M_2)$ with multiplicities at least  $m$.  Such divisors yield  an $\ell_m$-dimensional subspace of $\mathrm{H}^0(\mathbb{P}^1\times\mathbb{P}^1,\mathcal{O}_{\mathbb{P}^1\times\mathbb{P}^1}(-mK_{\mathbb{P}^1\times\mathbb{P}^1}))$. The effective divisors $\phi(D_1),\ldots, \phi(D_{\ell_m})$ induce a basis for this subspace.

 Let $G$ be an effective divisor on $\mathbb{P}^1\times \mathbb{P}^1$ that comes from $ |-mK_{S}|$.
 Let $L_1$ and $L_2$ be the two $0$-curves that pass through the point $\phi(p)$.
We write $G=m_1L_1+m_2L_2+\Omega$, where $m_1$ and $m_2$ are non-negative integers and $\Omega$ is an effective divisor whose support contains neither $L_1$ nor $L_2$.  We may assume that $\phi(M_1)$ belongs to $L_1$ and $\phi(M_2)$ belongs to $L_2$.

We use a bihomogeneous  coordinate system $([x:u], [y:v])$ for $\mathbb{P}^1\times\mathbb{P}^1$.  For the proof, we may assume that $\phi(p)=([0:1], [0:1])$.
Putting $u=1$ and $v=1$, we may regard $x$ and $y$ as local coordinates around the point $\phi(p)$. Even though we maps the surface $S$ onto  $\mathbb{P}^1\times \mathbb{P}^1$ instead of $\mathbb{P}^2$, the original method to estimate $c_p(S, \frac{1}{m\ell_m}D)$ works  verbatim for this case. 

The inequalities 
\[2m=L_1\cdot(m_1L_1+m_2L_2+\Omega)\geq m_2+(m-m_1)+\mult_{\phi(p)}(\Omega),\]
\[2m=L_2\cdot(m_1L_1+m_2L_2+\Omega)\geq m_1+(m-m_2)+\mult_{\phi(p)}(\Omega)\]
imply $|m_1-m_2|\leq m$ and  
\[m-|m_1-m_2| \geq \mult_\phi(p)(\Omega).\]

Note that $0\leq m_1, m_2\leq 2m$. This implies that the image of the injective map $\iota_m$ in \eqref{index-map2} can be arranged to be contained in the set  of the monomials corresponding to the integral  points in the shade area of the following diagram:
\medskip
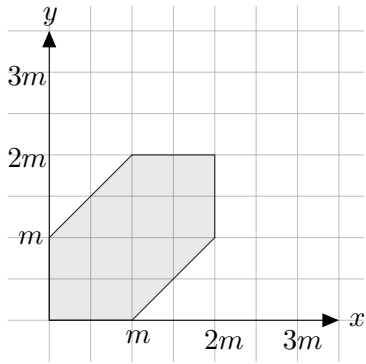
\begin{center}
\begin{tikzpicture}[scale=0.55, line cap=round,line join=round,>=triangle 45,x=0.2cm,y=0.2cm]
\draw [color=cqcqcq,, xstep=1.0cm,ystep=1.0cm] (-5.,-35.) grid (38.,8.);
\clip(-5.,-35.) rectangle (38.,8.);
\fill[color=sqsqsq,fill=sqsqsq,fill opacity=0.1] (0.,-30.) -- (0.,-20.) -- (10.,-10.) -- (20.,-10.) -- (20.,-20.) -- (10.,-30.) -- cycle;
\draw [->] (0.,-30.) -- (0.,5.);
\draw [->] (0.,-30.) -- (35.,-30.);
\draw (8,-30) node[anchor=north west] {$m$};
\draw (-5,-18) node[anchor=north west] {$m$};
\draw (17.5,-30) node[anchor=north west] {$2m$};
\draw (-6.3,-8) node[anchor=north west] {$2m$};
\draw (27,-30) node[anchor=north west] {$3m$};
\draw (-6.3,2) node[anchor=north west] {$3m$};
\draw (35,-28) node[anchor=north west] {$x$};
\draw (-2,9) node[anchor=north west] {$y$};
\draw [color=sqsqsq] (0.,-30.)-- (0.,-20.);
\draw [color=sqsqsq] (0.,-20.)-- (10.,-10.);
\draw [color=sqsqsq] (10.,-10.)-- (20.,-10.);
\draw [color=sqsqsq] (20.,-10.)-- (20.,-20.);
\draw [color=sqsqsq] (20.,-20.)-- (10.,-30.);
\draw [color=sqsqsq] (10.,-30.)-- (0.,-30.);
\end{tikzpicture}
\end{center}
 \medskip
Note that the number of monomials in this set is exactly $\ell_m$. Figure~\ref{table} shows that
$$\limsup_m \frac{m\ell_m}{v_m}=1.$$

From Cases~1 and~2, we can draw the conclusion that  $\delta (S)=1$.

\end{proof}

Now we consider the space $\mathbb{P}^1\times\mathbb{P}^1$. We use a bihomogeneous coordinate system $([x:u], [y:v])$ for $\mathbb{P}^1\times\mathbb{P}^1$. As before, fix a positive integer $m$ and set $\ell_m =h^0(\mathbb{P}^1\times\mathbb{P}^1, \mathcal{O}_{\mathbb{P}^1\times\mathbb{P}^1}(-mK_{\mathbb{P}^1\times\mathbb{P}^1}))$. The $\ell_m$ monomials of bidegree $(2m,2m)$ in variables $x, u;y,v$ form a basis for $\mathrm{H}^0(\mathbb{P}^1\times\mathbb{P}^1, \mathcal{O}_{\mathbb{P}^1\times\mathbb{P}^1}(-mK_{\mathbb{P}^1\times\mathbb{P}^1}))$. Therefore, the devisor defined by $xuyv=0$ is an anticanonical divisor of $m$-basis type.
This shows
\[\delta(\mathbb{P}^1\times\mathbb{P}^1)\leq 1.\]

\begin{theorem}
The $\delta$-invariant of $\mathbb{P}^1\times\mathbb{P}^1$ is $1$.
\end{theorem}
\begin{proof}
Let $\{s_1,\ldots,s_{\ell_m}\}$ be a basis of the space $\mathrm{H}^0(\mathbb{P}^1\times\mathbb{P}^1, \mathcal{O}_{\mathbb{P}^1\times\mathbb{P}^1}(-mK_{\mathbb{P}^1\times\mathbb{P}^1}))$. Denote the effective divisor  defined by the section $s_i$ by $D_i$. Set
$D=\sum D_i$. 
For a given point $p\in \mathbb{P}^1\times\mathbb{P}^1$, as before, the inequality $$c_p(\mathbb{P}^1\times\mathbb{P}^1, D)\geq \frac{1}{m\ell_m}$$
will be verified.

We may assume that $p=([0:1], [0:1])$.
Putting $u=1$ and $v=1$, we may regard $x$ and~$y$ as local coordinates around the point $p$.   As Case~2 in Theorem~\ref{theorem:degree 6},  the original method to estimate $c_p(\mathbb{P}^1\times\mathbb{P}^1, \frac{1}{m\ell_m}D)$ works  verbatim.

 Let $G$ be an effective divisor in  $ |-mK_{\mathbb{P}^1\times\mathbb{P}^1}|$.
 Let $L_1$ and $L_2$ be the two $0$-curves that pass through the point $p$.
We write $G=m_1L_1+m_2L_2+\Omega$, where $m_1$ and $m_2$ are non-negative integers and $\Omega$ is an effective divisor whose support contains neither $L_1$ nor $L_2$.

The inequalities 
\[2m=L_1\cdot(m_1L_1+m_2L_2+\Omega)\geq m_2+\mult_{p}(\Omega),\]
\[2m=L_2\cdot(m_1L_1+m_2L_2+\Omega)\geq m_1+\mult_{p}(\Omega)\]
imply  $0\leq m_1, m_2\leq 2m$ and  
\[2m-\max\{m_1, m_2 \}\geq \mult_p(\Omega).\]

This implies that the image of the injective map $\iota_m$ in \eqref{index-map2}  is  contained in the set  of the monomials corresponding to the integral  points in the shade area of the following diagram:
\medskip
\begin{center}
\begin{tikzpicture}[scale=0.55, line cap=round,line join=round,>=triangle 45,x=0.2cm,y=0.2cm]
\draw [color=cqcqcq,, xstep=1.0cm,ystep=1.0cm] (-5.,-35.) grid (38.,8.);
\clip(-5.,-35.) rectangle (38.,8.);
\fill[color=sqsqsq,fill=sqsqsq,fill opacity=0.1] (0.,-30.) -- (0.,-10.) -- (10.,-10.) -- (20.,-10.) -- (20.,-20.) -- (20.,-30.) -- cycle;
\draw [->] (0.,-30.) -- (0.,5.);
\draw [->] (0.,-30.) -- (35.,-30.);
\draw (8,-30) node[anchor=north west] {$m$};
\draw (-5,-18) node[anchor=north west] {$m$};
\draw (17.5,-30) node[anchor=north west] {$2m$};
\draw (-6.3,-8) node[anchor=north west] {$2m$};
\draw (27,-30) node[anchor=north west] {$3m$};
\draw (-6.3,2) node[anchor=north west] {$3m$};
\draw (35,-28) node[anchor=north west] {$x$};
\draw (-2,9) node[anchor=north west] {$y$};
\draw [color=sqsqsq] (0.,-30.)-- (0.,-10.);
\draw [color=sqsqsq] (0.,-10.)-- (10.,-10.);
\draw [color=sqsqsq] (10.,-10.)-- (20.,-10.);
\draw [color=sqsqsq] (20.,-10.)-- (20.,-20.);
\draw [color=sqsqsq] (20.,-20.)-- (20.,-30.);
\draw [color=sqsqsq] (20.,-30.)-- (0.,-30.);
\end{tikzpicture}
\end{center}
 \medskip
Figure~\ref{table} then shows that
$$\limsup_m \frac{m\ell_m}{v_m}=1.$$
\end{proof}

\section{Non-K-semistable del Pezzo surfaces}

In this section, we verify the last statement of Main Theorem. Since Conjecture~\ref{conjecture} has not been verified completely, at this moment we cannot say that the last statement of Main Theorem implies that the Hirzebruch surface  $\mathbb{F}_1$
and the del Pezzo surface of degree $7$ are not K-semistable. However, it is able to serve as a good evidence for Conjecture~\ref{conjecture}.

Let $[x:y:z]$ be a homogeneous coordinate for $\mathbb{P}^2$. The Hirzebruch surface $\mathbb{F}_1$ can be obtained by blowing up 
$\mathbb{P}^2$ at $p_1=[0:0:1]$. Let $\phi_1:\mathbb{F}_1\to \mathbb{P}^2$ be the blow-up with the exceptional divisor $E$. For a fixed positive integer $m$, the space $\mathrm{H}^0(\mathbb{F}_1,\mathcal{O}_{\mathbb{F}_1}(-mK_{\mathbb{F}_1}))$ can be regarded as the subspace $\mathcal{M}_1$ of $\mathrm{H}^0(\mathbb{P}^2,\mathcal{O}_{\mathbb{P}^2}(3m))$ consisting of the sections vanishing at $p_1$ with order at least~$m$. The set
\[\mathcal{B}_1=\left\{ x^ay^bz^c\ | \ a, b, c \mbox{ are non-negative integers with } a+b+c=3m, c\leq 2m \right\}\]
forms a basis for $\mathcal{M}_1$. Let $D_{a,b,c}$ be the divisor defined by $x^ay^bz^c$ on $\mathbb{P}^2$. 
Then the divisor
\[D=\frac{1}{m(2m+1)^2}\sum_{x^ay^bz^c\in\mathcal{B}_1} (\phi_1^*(D_{a,b,c})-mE)\]
is an anticanonical $\mathbb{Q}$-divisor of $m$-basis type on $\mathbb{F}_1$. The multiplicity of $D$ along the curve $E$ is~$\frac{7m+4}{6m+3}$. Therefore,
\[\delta_m(\mathbb{F}_1)\leq \frac{6m+3}{7m+4},\]
and hence
\[\delta(\mathbb{F}_1)\leq \frac{6}{7}.\]

We now let $S$ be the del Pezzo surface of degree $7$. It can be obtained by blowing up $\mathbb{P}^2$  at $p_1=[0:0:1]$ and $p_2=[0:1:0]$. Let $\phi_2:S\to \mathbb{P}^2$ be the blow-up with the exceptional curves  $E$ and $F$. For a fixed positive integer $m$, the space $\mathrm{H}^0(S,\mathcal{O}_{S}(-mK_{S}))$ can be regarded as the subspace $\mathcal{M}_2$ of $\mathrm{H}^0(\mathbb{P}^2,\mathcal{O}_{\mathbb{P}^2}(3m))$ consisting of the sections vanishing at $p_1$ and $p_2$ with order at least $m$. 
The set
\[\mathcal{B}_2=\left\{ x^ay^bz^c\ | \ a, b, c \mbox{ are non-negative integers with } a+b+c=3m, b\leq 2m, c\leq 2m\right\}\]
forms a basis for $\mathcal{M}_2$. Let $C_{a,b,c}$ be the divisor defined by $x^ay^bz^c$ on $\mathbb{P}^2$. 
The anticanonical $\mathbb{Q}$-divisor 
\[C=\frac{2}{7m^2(m+1)+2m}\sum_{x^ay^bz^c\in\mathcal{B}_2} (\phi_2^*(C_{a,b,c})-mE-mF)\]
is  of $m$-basis type on $S$. The multiplicity of $C$ along the proper transform of the curve defined by $x=0$ is $\frac{25m^2+27m+8}{21m(m+1)+6}$. Therefore,
\[\delta_m(S)\leq \frac{21m(m+1)+6}{25m^2+27m+8},\]
and hence
\[\delta(S)\leq \frac{21}{25}.\]

\bigskip
\bigskip

\section*{Appendix}

In Appendix Lemmas~\ref{lemma:injection1},~\ref{lemma:injection2} and~\ref{lemma:injection-after-coordinate-change} are verified. All the notations are the same as those in the beginning of Section~\ref{section:K-stable del Pezzo surfaces}.

\bigskip

We consider the vector space \[V_{\lambda}:=\bigoplus_{n=0}^{\lambda}k[x,y]_n\]
with an ordered basis $\{x^\alpha y^\beta\ | \ \alpha+\beta\leq \lambda\}$, where $k[x,y]_n$ is the $(n+1)$-dimensional vector space of homogeneous polynomials of degree $n$ in variables $x, y$. The order of the basis is given in the following way: 
\begin{enumerate}
\item the  graded lexicographic order with  $x\prec  y$ except for Case~2 in Theorem~\ref{theorem:d1} and Case~2 in Theorem~\ref{theorem:d5};
\item the  graded lexicographic order with  $y\prec  x$   for Case~2 in Theorem~\ref{theorem:d5};
\item the order for Case~2 in Theorem~\ref{theorem:d1} satisfies the properties:
\begin{enumerate}
\item $x^{\alpha_1} y^{\beta_1}\prec x^{\alpha_2} y^{\beta_2}$ if $\alpha_1+\beta_1<\alpha_2+\beta_2$;
\item $x^{\alpha} y^{\alpha}\prec x^{\alpha_2} y^{\beta_2}$ if $\alpha_2\ne \beta_2$ and $\alpha_2+\beta_2=2\alpha$.
\end{enumerate}
\end{enumerate}
Note that these orders  make  the monomial  $\prod_{i=1}^{d+1} \mathbf{x}_{1,i}^{n_i}$ in $\mathcal{C}_m$ smaller than any other monomials that appear in $\prod_{i=1}^{d+1} t_i^{n_i}$.

Since $f_{m, i}$ is a member of the vector space $V_{3m}$, we may express the polynomial $f_{m, i}$ as a~$1\times \sigma$ matrix with respect to the given ordered basis, where $\sigma=\frac{(3m+1)(3m+2)}{2}$. 
By writing these $1\times \sigma$ matrices as rows, we can express the $\ell_m$ polynomials $f_{m,1},\cdots, f_{m,\ell_m}$  altogether as a single $\ell_m\times \sigma$ matrix $M_F$. Since $f_{m,1},\cdots, f_{m,\ell_m}$ are linearly independent, the rank of the matrix $M_F$ is exactly~$\ell_m$. 

Let $E$ be a row echelon form of the matrix $M_F$. Then there is  an~$\ell_m\times \ell_m$ invertible matrix~$T$ such that $M_F=TE$. Since the rank of $M_F$ is $\ell_m$, the matrix~$E$ does not have any zero row.  The $i$-th row of $E$ represents a polynomial $h_{m,i}$ that belongs to $\mathcal{L}_m$. Its Zariski tangent term $t_{m,i}$ is represented by  the pivot  (the first non-zero entry from the left  in a row) and the entries whose corresponding  monomials have the same degree as  the monomial  corresponding to the pivot. In particular, the Zariski tangent term contains the monomial corresponding to  the pivot. 
The polynomials $h_{m,i}$ form a basis for the space $\mathcal{L}_m$ and their Zariski tangent terms~$t_{m, i}$ form a basis for the space $\mathcal{T}_m$. 

Since the monomial  $\prod_{i=1}^{d+1} \mathbf{x}_{1,i}^{n_i}$ in $\mathcal{C}_m$ is smaller than any other monomials that appear in~$\prod_{i=1}^{d+1} t_i^{n_i}$,  the set of the monomials corresponding to the columns with the pivots of $E$ must contain the set $\mathcal{C}_m$.

By collecting the $\ell_m$ pivot columns of $E$ in order, we obtain  an $\ell_m\times \ell_m$  upper triangular matrix  with the pivots on the diagonal. Denote this minor matrix of $E$  by $\widetilde{E}$. We also denote  the~$\ell_m\times \ell_m$  matrix $T\widetilde{E}$ by $\widetilde{M}_F$.  The entries of $i$-th row of $\widetilde{M}_F$ are the coefficients of the monomials in $f_{m, i}$ corresponding to the pivot columns of $E$. Since the matrix $\widetilde{M}_F$ is nonsingular, 
we can choose a single non-zero entry from each column  of  $\widetilde{M}_F$ in such a way that the non-zero entries are selected exactly one time from each row. This proves Lemmas~\ref{lemma:injection1} and~\ref{lemma:injection2}.

For Lemma~~\ref{lemma:injection-after-coordinate-change},  we consider the vector space $V_{\lambda}$ with a sufficiently large positive integer $\lambda$ so that we could write polynomials of bigger degrees  as matrices.
Note that the change of coordinate $x+A_1y^\beta \mapsto x$  in Step~C is given with $\beta>1$.
Since $$f_{m, i}(x,y)=\sum_{j=1}^{\ell_m}T_{ij}h_{m,j}(x,y)$$
for each $i$, where $T_{ij}$ is the entry of $T$ in the $i$-th row and the $j$-th column, we have 
$$f_{m, i}^{(1)}(x,y)=f_{m, i}(x-A_1y^\beta,y)=\sum_{j=1}^{\ell_m}T_{ij}h_{m,j}(x-A_1y^\beta,y).$$
Since $\beta>1$, we immediately see that the change of  coordinate 
does not give any effect on the Zariski  tangent term of $h_{m,i}$ at all. It therefore leaves  the positions of the pivot columns of $E$ unchanged. 
Therefore, by the same argument as for Lemmas~\ref{lemma:injection1} and~\ref{lemma:injection2}, we can obtain an injection in Lemma~~\ref{lemma:injection-after-coordinate-change}. Furthermore, the same argument works inductively  for $\{f^{(k)}_{m,i}\}$, $k>1$. This completes the proof of Lemma~~\ref{lemma:injection-after-coordinate-change}.

\end{document}